\pdfoutput=1
\documentclass[11pt]{amsart}
\usepackage[a4paper,margin=2cm,left=2cm,right=2cm]{geometry}

\usepackage[utf8]{inputenc}
\usepackage{algpseudocode}
\usepackage{amsmath}
\usepackage{amssymb}
\usepackage{amsthm}
\usepackage[english]{babel}
\usepackage{calligra}
\usepackage{cypriot}
\usepackage{dirtytalk}
\usepackage{enumitem}
\usepackage{fancyhdr}
\usepackage{float}
\usepackage{fdsymbol}
\usepackage{graphicx}
\usepackage{hyperref}
\usepackage{listings}
\usepackage{mathtools}
\usepackage{textcomp}
\usepackage{tikz}
\usetikzlibrary{arrows}
\usetikzlibrary{backgrounds,fit,positioning}
\usetikzlibrary{calc}
\usepackage{standalone}  
\usepackage{tikzscale}  
\usepackage{tipa}
\usepackage{xfrac}
\usepackage[counter,user]{zref}
\usepackage{mathdots}
\usepackage{todonotes}

\usepackage{csquotes,biblatex}
\theoremstyle{definition}
\newtheorem{theorem}{Theorem}[section]
\newtheorem{corollary}[section]{Corollary}

\newtheorem{definition}[theorem]{Definition}
\newtheorem{lemma}[theorem]{Lemma}
\newtheorem*{notation}{Notation}
\newtheorem*{mainthm}{Theorem}
\newtheorem*{mainlem}{Lemma}

\theoremstyle{remark}
\newtheorem{remark}[theorem]{Remark}
\newtheorem{example}[theorem]{Example}

\DeclareMathAlphabet{\mathpzc}{OT1}{pzc}{m}{it}
\DeclareFontShape{T1}{calligra}{m}{n}{<->s*[2.2]callig15}{}

\DeclareMathOperator*{\Flop}{Flop}
\DeclareMathOperator*{\Insert}{Insert}
\DeclareMathOperator*{\Cancel}{Cancel}

\makeatletter
\newcommand{\pushright}[1]{\ifmeasuring@#1\else\omit\hfill$\displaystyle#1$\fi\ignorespaces}
\newcommand{\pushleft}[1]{\ifmeasuring@#1\else\omit$\displaystyle#1$\hfill\fi\ignorespaces}
\makeatother

\newcommand{\Alg}{\text{Alg}}
\newcommand{\AlgList}{\text{AlgList}}

\newcommand{\del}{\partial}
\newcommand{\delt}{\tilde{\partial}}

\newcommand{\FO}{\widetilde{F_0}}
\newcommand{\Fn}{\widetilde{F_n}}
\newcommand{\Fnm}{\widetilde{F_{n+m}}}
\newcommand{\Fnn}{\widetilde{F_{2n}}}
\newcommand{\GO}{\widetilde{G_0}}
\newcommand{\Gm}{\widetilde{G_m}}
\newcommand{\id}{\text{id}}
\newcommand{\im}{\text{im}}
\newcommand{\Ind}{\text{Ind}}
\newcommand{\qand}{\quad\text{ and }\quad}

\newcommand{\morse}{\text{morse}}

\newcommand{\E}{\mathcal{E}}
\newcommand{\F}{\mathcal{F}}
\newcommand{\G}{\mathcal{G}}
\renewcommand{\P}{\mathbb{P}}

\newcommand{\M}{\mathcal{M}}
\newcommand{\N}{\mathbb{N}}
\newcommand{\R}{\mathbb{R}}
\newcommand{\RPtwo}{\mathbb{R}\mathbb{P}^2}

\newcommand{\Z}{\mathbb{Z}}

\title[The moduli space of flowlines in discrete Morse theory]{A combinatorial construction of the moduli space of flowlines in discrete Morse theory}
 
\author{Sophie Bleau}
\begin{document}

\maketitle
 
\begin{abstract}
    We construct the moduli space of index 2 flowlines of a discrete Morse function, giving a new proof showing that the Morse differential squares to zero in discrete Morse homology.
\end{abstract}
\section*{Introduction}
 
\subsection*{Context and results}

Given a manifold $M$ with a Morse function $f$, Morse theory defines a chain complex which in turn defines the Morse homology of $M$.
The method taken to prove that the square of the differential is zero in \cite{audin2014morse} is the following:
\begin{enumerate}[label = (\alph*)]
    \item Define the vector space of Morse chains on critical points of a Morse function $f:M\to\R$.
    \item Express the boundary map $\langle\del_n (x),y\rangle$ of a critical point $x$ and critical $(n-1)$ points $y$ by counting flowlines. This is done modulo 2.
    \item Show that the sum of coefficients of terms in the expression for $\langle\del^2(x),z\rangle$ is the count of broken trajectories.
     
    \item For the set of trajectories from an $(n+1)$-critical point to an $(n-1)$-critical point, demonstrate that we can include broken trajectories as the boundary of a 1-manifold of trajectories between two critical simplices. 
    \item Use the boundedness of the space to infer compactness, which gives the desired result (as the number of boundaries of a compact 1-manifold is even).
\end{enumerate}
For a simplicial complex with some discrete Morse function defined on its simplices satisfying constraints as defined in \ref{def:morsefunction}, we define a boundary map $\delt$ in terms of the flowlines (Definition \ref{def:flowline}) through the simplicial complex as shown in \ref{thm:d2=0}. 
We use the discrete version of Morse homology defined in \cite{forman} to give a proof that $\delt^2=0$, following a combinatorial analogue to that proven geometrically in \cite{audin2014morse}. 

We give a discrete analogue of the proof in \cite{audin2014morse}, doing so primarily by adapting step (d) in the above outline. We do so by defining an algorithm of evolving index 2 flowlines through the simplicial complex and looking at the signed count of boundary flowlines of the algorithm. 
The main results are Lemma \ref{lem:algInvolutivity}, Lemma \ref{lem:2critFsPerEquivClass} and Theorem \ref{thm:explicitCancellationOfFlowlines}. In particular, these state that this algorithm is involutive on critical flowlines, in the following sense. 
\begin{mainlem}[Restatement: Algorithm involutivity]\label{lem:algInvolutivityIntro}
    For a critical flowline, there is an algorithm that evolves it to another critical flowline. Furthermore, the action of the algorithm on critical flowlines is involutive.
\end{mainlem}
We see in Lemma \ref{lem:2critFsPerEquivClass} that a critical flowline has a unique distinct partner critical flowline, which has the following geometric interpretation. 
\begin{mainlem}[Restatement: Moduli space dimension]\label{lem:2critFsPerEquivClassIntro}
    For $\alpha$ an $(n+1)$-simplex and $\gamma$ an $(n-1)$-simplex, the moduli space of flowlines $\M(\alpha,\gamma)$ is a simplicial manifold of index 1.
\end{mainlem}
Finally, we conclude with Theorem \ref{thm:explicitCancellationOfFlowlines}, which shows us that the unique distinct critical flowline corresponding via the algorithm to some critical flowline must have an opposite sign.  
\begin{mainthm}[Restatement: Explicit cancellation of flowlines]\label{thm:explicitCancellationOfFlowlinesIntro}
The Morse differential squares to zero. 
\end{mainthm}

We have proved in this paper that $\delt^2=0$ in looking at the signed count of boundary flowlines of the algorithm. In the process of this, we have shown that the devised algorithm is involutive when acting on critical flowlines and that each boundary flowline has a unique distinct boundary flowline in its equivalence class. Furthermore, although \cite{forman} requires that Hasse diagrams of the simplicial complex with Morse function cannot have cyclic paths, the results here are independent of such an assumption. This is to say, we do not require our gradient vector fields to be discrete as in the sense of Theorem \ref{thm:no_cycles}. The result in \cite{audin2014morse} uses the count of objects in modulo 2, whereas in this paper we will give each arrow a sign as defined in \cite{MorseQuivers} and show that the signed count of objects is zero. Further study may involve the analysis of flowlines with more than two drops in dimension, although, as recognized in [p2-3, \cite{nanda}], ordering the two double drops and negotiating how they might `pass' one another becomes a problem of its own.

\subsection*{Main idea}
Given $f:K\to\R$ a simplicial Morse function, a flowline through a simplicial complex consists of a sequence of simplices decreasing in Morse value, such that each step in the sequence either increases or decreases the simplex dimension by one, and we can never increase the dimension two steps in a row.
 
For instance, for a triangle (known as a 2-simplex) we can have a path of simplices to one of its vertices (known as 0-simplices) via an adjacent edge. 
\begin{figure}[H]
    \centering
\begin{tikzpicture}

\node (v1) at (-2,-1) {$\gamma$};
\node[draw, fill, minimum size=1, inner sep=1, outer sep=1] (v3) at (2,-1) {};
\node[draw, fill, minimum size=1, inner sep=1, outer sep=1] (v2) at (0,2) {};
\draw  (v1) edge (v2) ;
      \draw (-1,-1)--(1,-1)
      node[midway, below] {$\beta$};
\draw  (v1) edge (v3);
\draw  (v3) edge (v2);
\node at (0,0.2) {$\alpha$};
\draw [-stealth][red] plot[smooth, tension=.0] coordinates {(0,0) (0,-1) (-1.7691,-0.9976)};
\end{tikzpicture}
    \caption{The figure shows a 2-simplex, $\alpha$, with 1-simplices as its edges and 0-simplices as its vertices. A path from the 2-simplex to one of its 0-simplices, $\gamma$ via boundary arrows is shown in red.}
    \label{fig:intro2simp}
\end{figure}
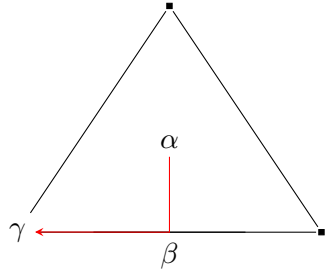
 
In this example, we want to show that the number of paths from $\alpha$ to $\gamma$ is 0 modulo 2; in this setting, there is an opposite path via the other edge. We will later see that they have opposite signs, so we may consider the first flowline with coefficient 1 and the second with coefficient $-1$ such that, considering sign, there are 0 flowlines.

\begin{figure}[H]
    \centering
\begin{tikzpicture}
\node (v1) at (-2,-1) {$\gamma$};
\node [draw, fill, minimum size=1, inner sep=1, outer sep=1](v3) at (2,-1) {};
\node [draw, fill, minimum size=1, inner sep=1, outer sep=1](v2) at (0,2) {};
\draw  (v1) edge (v2) ;
      \draw (-1,-1)--(1,-1)
      node[midway, below] {$\beta$};
\draw  (v1) edge (v3);
\draw  (v3) edge (v2);
\node at (0,0) {$\alpha$};
\draw [-stealth][red] plot[smooth, tension=.0] coordinates {(-0.001,-0.2178) (0,-1) (-1.7691,-0.9976)};
\draw [-stealth][blue]plot[smooth, tension=0] coordinates {(-0.1773,0.1828) (-0.8,0.8) (-1.8738,-0.8112)};
\end{tikzpicture}
    \caption{The figure shows an alternative path from the 2-simplex to the 0-simplex in blue.}
    \label{fig:intro2simp2}
\end{figure}
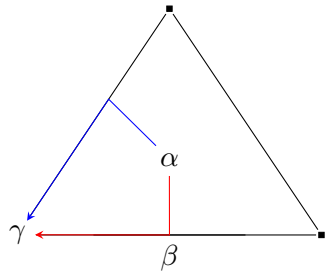

More generally, any flowline consisting of a double drop in dimension has a unique partner with a double drop point, as shown in the diagram below. We will formalise taking a double drop point to its alternative double drop, and call it the `Flop' operation. 
\begin{figure}[H]
    \centering
    
\begin{tikzpicture}
\node (v1) at (0,1) {$\alpha$};
\node (v3) at (0,-1) {$\gamma$};
\node (v2) at (-1,0) {$\beta$};
\node (v4) at (1,0) {$\beta'$};
\draw  [-stealth](v1) edge (v2);
\draw  [-stealth](v2) edge (v3);
\draw  [-stealth][dotted](v1) edge (v4);
\draw  [-stealth][dotted](v4) edge (v3);
\draw  [-stealth][dashed](v2) edge (v4);
\node at (0.0,0.2) {$\exists !$};
\end{tikzpicture}

    \caption{The figure shows that between some $(n+1)$-simplex, $\alpha$, and some $(n-1)$-simplex in its boundary, $\gamma$, there exist two distinct paths $\alpha\to\gamma$. The bold arrows $\alpha\to\beta\to\gamma$ represent the original path, the dotted arrows $\alpha\dashrightarrow\beta'\dashrightarrow\gamma$ represent the unique alternative path to the original path, and the dashed arrow represents the unique action of taking one path to the other.}
    \label{fig:introExistFlop}
\end{figure}
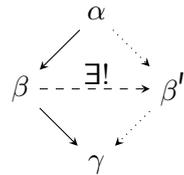

Our result generalises this finding to longer paths. Some flowlines, called critical flowlines, are special because of the type of simplex at this double drop in dimension. We wish to show in this paper that for a critical flowline, there is a unique alternative critical flowline. This is on a larger scale, and each edge in the diagram may consist of a flowline of multiple rises and drops in dimension, as shown. Note that, unlike the alternative drop point, $\beta'$ could be equal to $\beta$, but we require the critical flowlines to be distinct.
\begin{figure}[H]
    \centering
    \begin{tikzpicture}
\node (v1) at (0,1) {$\alpha$};
\node (v3) at (0,-1) {$\gamma$};
\node (v2) at (-1,0) {$\beta$};
\node (v4) at (1,0) {$\beta'$};
\draw [-stealth]plot[smooth, tension=.0] coordinates {(-0.2,0.9) (-0.2,0.7) (-0.4,0.7) (-0.4,0.5) (-0.6,0.5) (-0.6,0.3) (-0.8,0.3) (-0.8,0.1)};
\draw [-stealth]plot[smooth, tension=.0] coordinates {(-0.9,-0.1) (-0.9,-0.3) (-0.7,-0.3) (-0.7,-0.5) (-0.5,-0.5) (-0.5,-0.7) (-0.3,-0.7) (-0.3,-0.9) (-0.1,-0.9)};
\draw  [-stealth][dotted]plot[smooth, tension=.0] coordinates {(0.2,0.9) (0.2,0.7) (0.4,0.7) (0.4,0.5) (0.6,0.5) (0.6,0.3) (0.8,0.3) (0.8,0.2)};
\draw  [-stealth][dotted]plot[smooth, tension=.0] coordinates {(0.1,-0.9) (0.3,-0.9) (0.3,-0.7) (0.5,-0.7) (0.5,-0.5) (0.7,-0.5) (0.7,-0.3) (0.9,-0.3) (0.9,-0.1)};
\draw  [-stealth][dashed](v2) edge (v4);
\node at (0.0,0.2) {$\exists !$};
\end{tikzpicture}
    \caption{The figure shows an analogue to the double drop on a larger scale. It shows that for a critical flowline $\alpha\to \gamma$, there is another critical flowline we can find via an algorithm. The zig-zag arrows represent the steps up and down in dimension from an $n$-simplex to an $(n-1)$-simplex. The dotted arrow $\alpha\dashrightarrow\beta'\dashrightarrow\gamma$ is the unique alternative flowline to $\alpha\to\beta\to\gamma$, and the dashed arrow again represents the unique action of taking one path to the other.}
    \label{fig:introFlowlineFlop}
\end{figure}
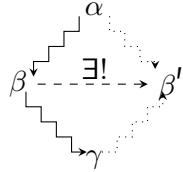

We will find in the course of the paper that the critical flowlines are the two endpoints of a one-dimensional space of flowlines defined by some algorithm. In this way, we will show that the critical flowlines come in pairs, and thus that the discrete Morse differential squares to zero.

\subsection*{Outline}
In Section \ref{sec:Complexes} we will discuss simplicial complexes with an abstract example. This leads us into Section \ref{sec:GradientVectorFields} where simplicial complexes are endowed with gradient vector fields in the form of Morse functions. Here we also discuss the definition of $\delt$, as well as flowlines through a simplicial complex and the signs induced on them. Section \ref{sec:The fundamental lemma of Morse theory} contains the main meat of the paper. We show an example of the canonical Morse function on a simplicial complex and prove that $\delt^2=0$ in this case. Then we define our algorithm on flowlines through a simplicial complex with Morse function, showing two examples of how this algorithm functions.
In Section \ref{sec:Proving that the boundary square must have a pair} we show that the algorithm is involutive and that the critical flowlines come in opposite pairs which cancel out in sign, giving $\delt^2=0$.

\subsection*{Acknowledgements}
I would like to thank Jeff Hicks for his supervision throughout this project, as well as Diana Bergerova for her support, comments and feedback. The project was funded by the University of Edinburgh School of Mathematics and College of Science and Engineering. 

\section{Background: Complexes}\label{sec:Complexes}
In this paper, we wish to study the universe of simplices. Simplices are the building blocks of a simplicial complex, and they have certain properties we are interested in.

A vertex is a $0$-simplex, an edge is a $1$-simplex, a triangle is a $2$-simplex and a tetrahedron is a $3$-simplex. In your head, you may be able to justify this based on the number of vertices in each of the sets, and more importantly, the dimensions of the spaces in which these shapes lie, but let's give a more formal definition.

Let $V$ be a set called the set of vertices. For subsets $\tau,\sigma$ of $V$, let $\tau$ be a `face' of $\sigma$ if and only if all subsets of $\tau$ are subsets of $\sigma$. 
For example
\[\{(1)\} \text{ is a face of } \{(1,2,3)\}.\]
Let $\tau$ be a `facet' of $\sigma$ if $\tau$ is a face of $\sigma$ and $\tau$ has one fewer element than $\sigma$ (which is to say that if $\sigma$ is $n$-dimensional then $\tau$ is $(n-1)$ dimensional), so $\tau$ is a maximal face of $\sigma$. For example
\[\{(1,2)\} \text{ is a facet of } \{(1,2,3)\}.\]

\begin{definition}[Simplicial complex]
A \textbf{simplicial complex} $K$ on vertices $V$ is a set of subsets $\sigma \subset V$ such that if $\sigma\in K$ and $\tau$ is a face of $\sigma$ then $\tau\in K$.
\end{definition}
This is to say that, for example, a tetrahedron cannot be in the simplicial complex if not all its faces are.

A subcomplex of a simplicial complex $K$ is a collection of some of these objects in a $K$ such that the constraints are still satisfied. Now let's define the objects we've been talking so vaguely about. 
 
\begin{definition}[Simplex]
    For $K$ a simplicial complex, an $n$-\textbf{simplex} $S$ is a subcomplex of $K$ consisting of all of the subsets of some face $\sigma\in K$.
     
    Its boundary is populated by $(n-1)$-simplices. 
\end{definition}

\begin{notation}
    For an $n$-simplex $\alpha$, we often write the index $\alpha^{(n)}$ to indicate dimension.
\end{notation}
We can store all the information from a simplicial complex in one diagram.

\subsection{The Hasse diagram}
Let $K$ be a simplicial complex where $p$ is the dimension of the largest simplex. Arrange the $p$-dimensional simplices in a row, the $(p-1)$ dimensional simplices in the next row, and carry on until we have $p+1$ rows (the bottom such being the collection of vertices). For each simplex, draw arrows called boundary arrows connecting it to each of its facets. The resulting directed graph is called the \textbf{Hasse diagram} of $K$.
For example, the following simplicial complex (left) has the associated Hasse diagram (right).
\begin{figure}[H]
    \centering
    \begin{tikzpicture}
\node (v1) at (-4.2,0.8) {1};
\node (v2) at (-5.4,-1) {2};
\node (v3) at (-3,-1) {3};
\draw  (v1) edge (v2);
\draw  (v2) edge (v3);
\draw  (v3) edge (v1);
\node (v4) at (-4.2,-2.8) {4};
\draw  (v2) edge (v4);
\draw  (v4) edge (v3);
\node at (-4.2,-0.4) {123};
\node at (-4.2,-1.6) {234};
\node (v1) at (0,0) {123};
\node (v8) at (2,0) {234};
\node (v4) at (1,-1) {23};
\node (v3) at (0,-1) {13};
\node (v2) at (-1,-1) {12};
\node (v9) at (2,-1) {24};
\node (v10) at (3,-1) {34};
\node (v5) at (-1,-2) {1};
\node (v6) at (0.3,-2) {2};
\node (v7) at (1.7,-2) {3};
\node (v11) at (3,-2) {4};
\draw  [-stealth](v1) edge (v2);
\draw  [-stealth](v1) edge (v3);
\draw  [-stealth](v1) edge (v4);
\draw  [-stealth](v2) edge (v5);
\draw  [-stealth](v2) edge (v6);
\draw  [-stealth](v3) edge (v5);
\draw  [-stealth](v3) edge (v7);
\draw  [-stealth](v4) edge (v6);
\draw  [-stealth](v4) edge (v7);
\draw  [-stealth](v8) edge (v4);
\draw  [-stealth](v8) edge (v9);
\draw  [-stealth](v8) edge (v10);
\draw  [-stealth](v9) edge (v6);
\draw  [-stealth](v9) edge (v11);
\draw  [-stealth](v10) edge (v7);
\draw  [-stealth](v10) edge (v11);
\end{tikzpicture}
    \caption{The figure shows on the left a simplicial complex in which each 2-simplex is labelled by its vertices. On the right, we see the Hasse diagram associated with the simplicial complex by drawing an arrow from an $n$-simplex to all its $(n-1)$ dimensional facets.}
    \label{fig:Hasse1}
\end{figure}
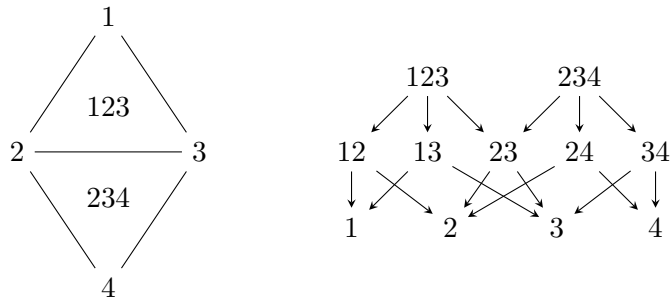

As observed in the introduction, for any path from an $(n+1)$-simplex to an $(n-1)$-simplex, we must have a unique alternative path between these simplices. We will prove this in Lemma \ref{lem:floppability assertion}.
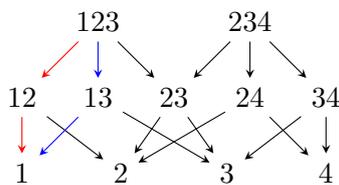
\begin{figure}[H]
    \centering
    \begin{tikzpicture}
\node (v1) at (0,0) {123};
\node (v8) at (2,0) {234};
\node (v4) at (1,-1) {23};
\node (v3) at (0,-1) {13};
\node (v2) at (-1,-1) {12};
\node (v9) at (2,-1) {24};
\node (v10) at (3,-1) {34};
\node (v5) at (-1,-2) {1};
\node (v6) at (0.3,-2) {2};
\node (v7) at (1.7,-2) {3};
\node (v11) at (3,-2) {4};
\draw  [red][-stealth](v1) edge (v2);
\draw  [blue][-stealth](v1) edge (v3);
\draw  [-stealth](v1) edge (v4);
\draw  [red][-stealth](v2) edge (v5);
\draw  [-stealth](v2) edge (v6);
\draw  [blue][-stealth](v3) edge (v5);
\draw  [-stealth](v3) edge (v7);
\draw  [-stealth](v4) edge (v6);
\draw  [-stealth](v4) edge (v7);
\draw  [-stealth](v8) edge (v4);
\draw  [-stealth](v8) edge (v9);
\draw  [-stealth](v8) edge (v10);
\draw  [-stealth](v9) edge (v6);
\draw  [-stealth](v9) edge (v11);
\draw  [-stealth](v10) edge (v7);
\draw  [-stealth](v10) edge (v11);
\end{tikzpicture}
    \caption{The figure shows within the Hasse diagram the two paths from the 2-simplex 123 to 1, a 0-dimensional face.}
    \label{fig:FlopInHasse}
\end{figure}

We will discuss more about why this is and what this means in \S\ref{subsec:CanonicalMorseFunction}. But first, let's see an example of a simplicial complex which does not arise from topology.

\subsection{An example of colourable complexes}
Let $K_{k,n}$ be all the graphs on $n$ vertices which are $k$-edge-colourable
\footnote{We say a graph is $k$-edge-colourable if we can colour its edges with $k$ colours in such a way that no two adjacent edges are coloured the same.} (not necessarily connected). For $n=3$, $k=2$ we find that only the graphs with $\leq 2$ edges are $k$-colourable, so the complete graph is not 2-colourable.

\begin{figure}[H]
    \begin{center}
        \includegraphics[width=0.5\textwidth]{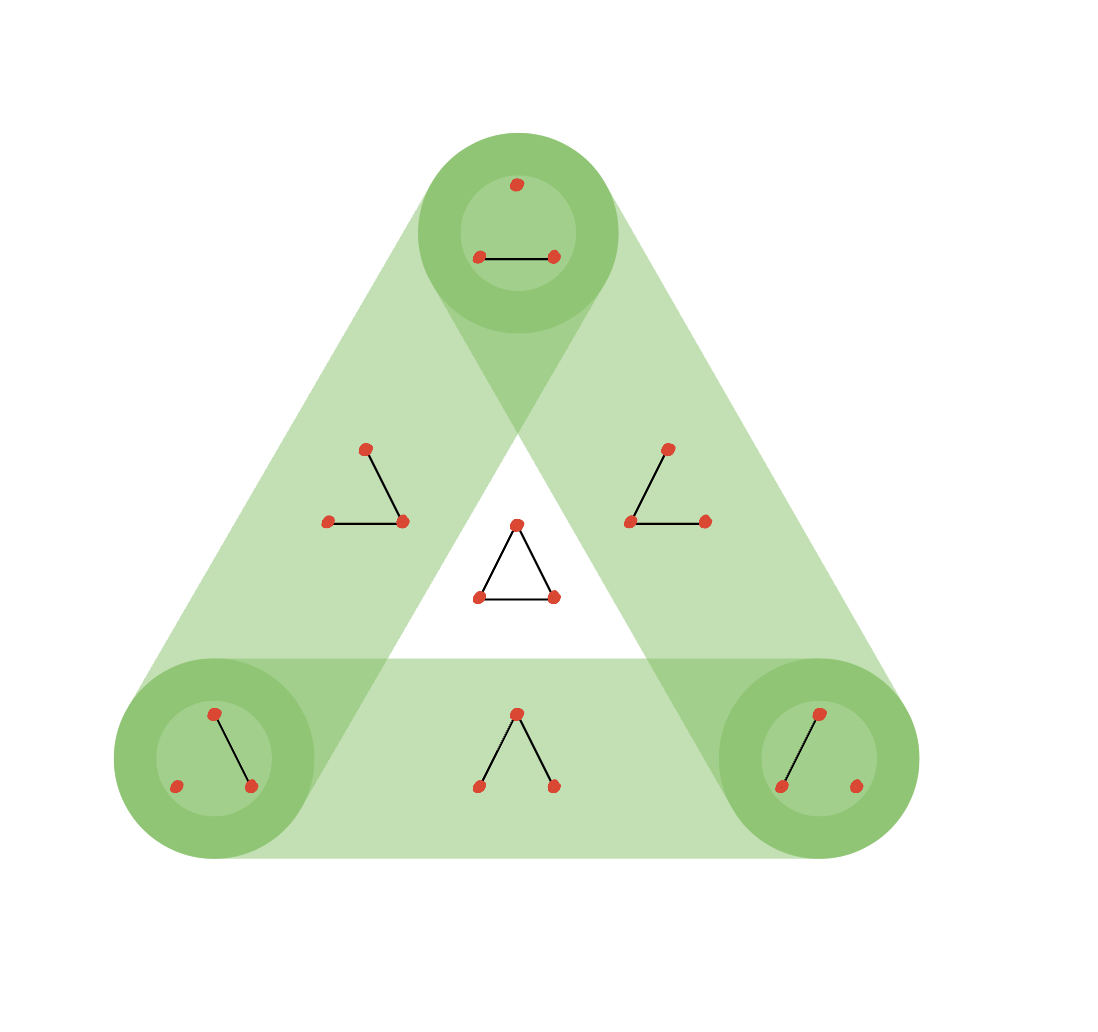}
        \caption{The figure shows the $2$-colourable subgraphs of a triangle, where the lower dimensional subcomplexes consist of one edge and the higher consist of two edges. The simplicial complex, shown in green, has these subgraphs as its simplices. Note that the subgraph with three edges is not $2$-colourable. }
        \label{fig:tri2col}
    \end{center}
\end{figure}
 
We use the data of $K_{k,n}$ to construct an abstract simplicial complex. The vertices of this complex will be graphs with a single edge. The vertices $\{e_0,e_1,\dots,e_d\}$ belong to a the $d$-dimensional simplices if the corresponding graphs with $d+1$ edges is $k$-edge-colourable. Here, two one-edged graphs are connected by a two-edged graph, but there are no $2$-simplices since there are no three-edged 2-edge-colourable graphs with 3 vertices.

In fact, $k$-edge-colourability is an example of a \textbf{graph property}.
A graph property is called \textbf{monotone decreasing} if for any spanning graphs $G_1\subset G_2$ if $G_2$ has the graph property then $G_1$ must too.
Some other examples of monotone decreasing graph properties are the following: graphs with $\leq k$ edges; graphs $G$ with $\deg(v)\leq d$ for all $v\in G$ and for some $d\in \N$; not connected graphs; not $i$-connected graphs; graphs with no Hamiltonian cycles; and bipartite graphs. Each monotone decreasing graph property gives rise to an abstract simplicial complex.

\section{Background: Gradient vector fields}\label{sec:GradientVectorFields}
We recall the definition of a Morse function on a simplicial complex, define an orientation on a simplex, and express a simplicial complex endowed with a Morse function in a `modified Hasse diagram'. Morse functions are functions defined on a simplicial complex $K$, assigning to each simplex $\sigma\in K$ some value.
We impose the following constraint: we ask that higher dimensional simplices ``usually'' have higher value in Morse function. This means that sliding downwards in Morse value corresponds to sending an $n$-simplex to one of its facets via a boundary arrow, sliding down dimension with the help of gravity.
The idea of a simplex in a simplicial complex being \textbf{critical} under a Morse function $f$ is to say that it is a well-behaved simplex of the function. 
In defining a Morse function, we require simplices of lower dimension than its neighbours to have a larger image in $f$ than at most one of its higher dimensional neighbours, and those of higher dimension than its neighbours to have a smaller image than at most one of its neighbours. Let us state this more formally.
\begin{definition}\label{def:morsefunction}
    We call $f:K\to \R$ a discrete Morse function if for each simplex $\beta^{(p)}$ in $K$, 
    \begin{itemize}
        \item at most one neighbouring lower dimensional simplex $\gamma^{(p-1)}$ has a higher Morse assignment, i.e.
        \[\#\left\{\gamma^{(p-1)}<\beta | f(\gamma)\geq f(\beta)\right\}\leq 1,\]
        \item at most one neighbouring higher dimensional simplex $\alpha^{(p+1)}$ has a lower Morse assignment, i.e.
        \[\#\left\{\alpha^{(p+1)}>\beta | f(\alpha)\leq f(\beta)\right\}\leq 1,\]
    \end{itemize}
\end{definition}
The most ``well-behaved'' a simplex can be with respect to $f$ is \textbf{critical} (I like to remember that it's so well-behaved it thinks it has the right to criticise all the other simplices). A critical simplex is one that has no occurrences of higher dimensional simplices being of lower Morse value, or lower dimensional simplices being of higher Morse value. This is to say that the inequalities in the above bullet points must be strict.  
\begin{example}
The canonical Morse function is one where $f(\sigma)=\dim(\sigma)$ for all simplices $\sigma\in K$. A simplicial complex with the canonical Morse function is shown in Figure \ref{fig:SimpCompMorseCrit}, in which every simplex is critical. 
\begin{figure}[H]
    \centering
    \begin{tikzpicture}
\node (v1) at (0,2) {0};
\node (v2) at (-1.2,0.2) {0};
\node (v3) at (1.2,0.2) {0};
\draw  (v1) edge (v2);
\draw  (v2) edge (v3);
\draw  (v3) edge (v1);
\node (v4) at (0,-1.6) {0};
\draw  (v2) edge (v4);
\draw  (v4) edge (v3);
\node at (0,1) {2};
\node at (0,-0.4) {2};
\node at (-0.8,1.3) {1};
\node at (0.7,1.3) {1};
\node at (0.7,-0.9) {1};
\node at (-0.7,-0.9) {1};
\node at (0,0.4) {1};
\end{tikzpicture}
    \caption{The figure shows a simplicial complex on which the canonical Morse function is applied. The value on each simplex is the value of the discrete Morse function at that simplex. Every simplex has a higher Morse value than its lower dimensional neighbours, and a lower Morse value than its higher dimensional neighbours. Every simplex in this simplicial complex is therefore critical.}
    \label{fig:SimpCompMorseCrit}
\end{figure}
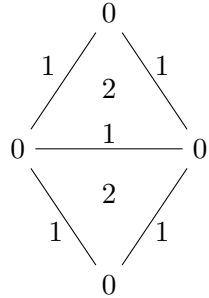

As we can see, higher dimensional simplices have a higher Morse function output, and vice versa. 
\end{example}
\begin{example}
A Morse function on a simplicial complex with only one critical simplex is shown in Figure \ref{fig:nonCritSimpCompEx}.
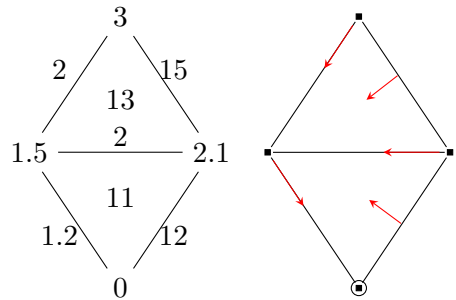
\begin{figure}[H]
    \centering
\begin{tikzpicture}
\node (v1) at (-2.2,1.8) {3};
\node (v2) at (-3.4,0) {1.5};
\node (v3) at (-1,0) {2.1};
\draw  (v1) edge (v2);
\draw  (v2) edge (v3);
\draw  (v3) edge (v1);
\node (v4) at (-2.2,-1.8) {0};
\draw  (v2) edge (v4);
\draw  (v4) edge (v3);
\node at (-2.2,0.7) {13};
\node at (-2.2,-0.6) {11};
\node at (-3,1.1) {2};
\node at (-1.5,1.1) {15};
\node at (-1.5,-1.1) {12};
\node at (-3,-1.1) {1.2};
\node at (-2.2,0.2) {2};
\node [draw, fill, minimum size=1, inner sep=1, outer sep=1] (v1) at (0.9381,1.7987) {};
\node [draw, fill, minimum size=1, inner sep=1, outer sep=1] (v2) at (-0.2619,-0.0013) {};
\node [draw, fill, minimum size=1, inner sep=1, outer sep=1] (v3) at (2.1381,-0.0013) {};
\draw  (v1) edge (v2);
\draw  (v2) edge (v3);
\draw  (v3) edge (v1);
\node [draw, fill, minimum size=1, inner sep=1, outer sep=1] (v4) at (0.9381,-1.8013) {};
\draw  (v2) edge (v4);
\draw  (v4) edge (v3);
\node at (0.9381,0.5987) {};
\node at (0.9381,-0.6013) {};
\draw  [red][stealth-]plot[smooth, tension=.7] coordinates {(1.0381,0.6947) (1.4478,1.0154)};
\draw  [red][stealth-]plot[smooth, tension=.7] coordinates {(1.0737,-0.6324) (1.4923,-0.9441)};
\draw  [red][-stealth]plot[smooth, tension=.7] coordinates {(-0.191,-0.1069) (0.2097,-0.7125)};
\draw  [red][-stealth]plot[smooth, tension=.7] coordinates {(0.8599,1.6834) (0.4859,1.1134)};
\draw  [red][-stealth]plot[smooth, tension=.7] coordinates {(2,0) (1.2608,0.0001)};
\draw  (v4) ellipse (0.1 and 0.1);
\end{tikzpicture}
    \caption{The figure shows first a simplicial complex with a nontrivial Morse function, and second the same simplicial with no Morse values: only the information of the Morse arrows. The critical simplex is circled in the second diagram.}
    \label{fig:nonCritSimpCompEx}
\end{figure}

Every instance of a lower dimensional simplex having a higher Morse function output or vice versa is indicated in the second diagram by a red arrow. This indicates that we can still travel downwards: not in the traditional sense of dimension, but in Morse value. 
\end{example}

The topic of \textbf{gradient vector fields} intersects discrete Morse theory by paying attention to the instances of bad behaviour, which is to say boundary arrows between noncritical simplices.
 
For instance, when we have a lower dimensional simplex $\gamma^{(p-1)}$ than $\alpha^{(p)}$ with $f(\gamma)\geq f(\alpha)$, we include the pair of simplices $\{\gamma, \alpha\}$ in the gradient vector field. To illustrate this pairing on a simplicial complex we may draw an arrow from the lower dimensional simplex $\gamma$ to the higher, $\alpha$, as in Figure \ref{fig:modHasse1}. As the boundary arrow naturally takes us down a dimension each time we use it, drawing an arrow $\gamma\to\alpha$ is a clear way of indicating that we're travelling against gravity. 

\begin{definition}[Definition 3.3 \cite{forman}]
    A discrete vector field $V$ on $K$ is a collection of pairs $\{\alpha^{(p)}<\beta^{(p+1)}\}$ of simplices of $K$ such that each simplex is in at most one pair of $V$.
\end{definition}

\begin{lemma}\label{lem:headtailorneither}
    In order to obey the strict rules given in Definition \ref{def:morsefunction}, every simplex must be exactly one of;
    \begin{itemize}
        \item the head of an arrow in $V$,
        \item the tail of an arrow in $V$, or
        \item neither the head nor tail of an arrow in $V$ (a critical simplex).
    \end{itemize}
    The pairing induced by a discrete Morse function is called a discrete gradient vector field.
\end{lemma}

Usually we travel down dimensions via boundary arrows, but given that a discrete vector field takes a lower dimensional simplex to a higher one, we indicate each pairing from the discrete vector field with a `Morse arrow' drawn on the simplicial complex. 

The trivial case is when the discrete vector field is empty, which is to say that there are no arrows against the grain, and all simplices are critical. At the other extreme, we say $V$ is a \textbf{complete matching} when every simplex belongs to exactly one pair in $V$, which is to say that there are no critical simplices. One of the goals of discrete Morse theory is to generate Morse functions on a simplicial complex which minimise the critical set - this is discussed in \cite{sharko} and a specific case of this is summarised succinctly in \cite{forman}. 

A discrete vector field allows us to traverse the simplicial complex by travelling simplex to simplex: alternating between climbing up Morse arrow maps and sliding down the boundary arrows. A $\mathbf{V}$-\textbf{path} is such a journey. We define it to be a sequence of simplices
\[\alpha_0^{(p)},\beta_0^{(p+1)},\alpha_1^{(p)},\beta_1^{(p+1)},\alpha_2^{(p)},\dots,\beta_r^{(p+1)},\alpha_{r+1}^{(p)}\] where each $\{\alpha_i<\beta_i\}\in V$ and $\beta_i>\alpha_{i+1}$. We then have the following theorem on the behaviour of a discrete Morse function $f$ on a $V$-path.

\begin{theorem}[Theorem 3.4, \cite{forman}]
For a discrete Morse function $f$ with an associated gradient vector field $V$, a sequence of simplices is a $V$-path if and only if $\alpha_i<\beta_i> \alpha_{i+1}$ for each $i\in\{0,\dots,r\}$, and
\[f(\alpha^0)\geq f(\beta^0)>f(\alpha^1)\geq f(\beta^1)>\dots\geq f(\beta^r)>f(\alpha^{r+1}).\]
\end{theorem}

\begin{theorem}[Theorem 3.5, \cite{forman}]\label{thm:no_cycles}
    A discrete vector field $V$ is the gradient vector field of a discrete Morse function if and only if there are no non-trivial closed $V$-paths.
\end{theorem}

\begin{theorem}[Theorem 3.6, \cite{forman}]
    For a directed graph $G$, there is a real-valued function of the vertices that is strictly decreasing along each directed path if and only if there are no directed loops. 
\end{theorem}

\begin{remark}
    We find in the following arguments that we in fact have no need to assume that $V$ is cycle-free for our later assertions. For expositional purposes we work with gradient vector fields - however, every instance in the following sections may be replaced with discrete vector fields and the results still hold.
\end{remark}

\subsection{The Hasse diagram of a gradient vector field}
We can encode the information of a gradient vector field on a simplicial complex into a diagram in the following way. 
Given a simplicial complex $K$, and a discrete vector field $V$, for each pair $\{\alpha<\beta\}\in V$ reverse the arrow to point upwards. This is the \textbf{modified Hasse diagram} corresponding to $V$. A $V$-path is a directed path in this diagram whose starting simplex has the same dimension as its finishing index.
\begin{figure}[H]
    \centering
    \begin{tikzpicture}
\node (v1) at (0,0) {123};
\node (v8) at (2,0) {234};
\node (v4) at (1,-1) {23};
\node (v3) at (0,-1) {13};
\node (v2) at (-1,-1) {12};
\node (v9) at (2,-1) {24};
\node (v10) at (3,-1) {34};
\node (v5) at (-1,-2) {1};
\node (v6) at (0.3,-2) {2};
\node (v7) at (1.7,-2) {3};
\node (v11) at (3,-2) {4};
\draw  [-stealth](v1) edge (v2);
\draw  [red][stealth-](v1) edge (v3);
\draw  [-stealth](v1) edge (v4);
\draw  [red][stealth-](v2) edge (v5);
\draw  [-stealth](v2) edge (v6);
 
\draw  [-stealth](v3) edge (v5);
\draw  [-stealth](v3) edge (v7);
\draw  [-stealth](v4) edge (v6);
\draw  [red][stealth-](v4) edge (v7);
\draw  [-stealth](v8) edge (v4);
 
\draw  [-stealth](v8) edge (v9);
\draw  [red][stealth-](v8) edge (v10);
\draw  [red][stealth-](v9) edge (v6);
\draw  [-stealth](v9) edge (v11);
\draw  [-stealth](v10) edge (v7);
\draw  [-stealth](v10) edge (v11);
\end{tikzpicture}
    \caption{The figure shows the modified Hasse diagram of the simplicial complex, where the Morse arrows are displayed in red. Note that the simplex labelled `4' is critical, as there are no Morse arrows adjacent to it. An example of a $V$-path is $13\to 123\to 23\to 2\to 24\to 4$.}
    \label{fig:modHasse1}
\end{figure}
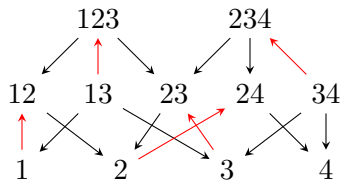

As we travel around the simplicial complex in this way, we would particularly like to pass a critical simplex on the way for reasons I shall explain. 
 
To formalise our journey through the simplicial complex, let us define a few concepts. 
\begin{definition}[A.2, \cite{MorseQuivers}]\label{def:flowline}
    A \textbf{flowline}\footnote{A flowline is also often called a \textbf{gradient flow trajectory} in other texts.} is a path between critical simplices $\alpha$ and $\gamma$ which follows the arrows given in the modified Hasse diagram for the simplicial complex. 
\end{definition}

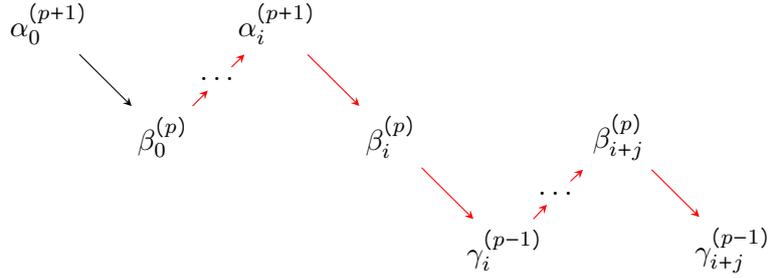
\begin{figure}[H]
    \begin{center}
        \begin{tikzpicture}
        \node (v1) at (-1.5,1.5) {$\alpha_0^{(p+1)}$};
        \node (v2) at (0,0) {$\beta_0^{(p)}$};
        \node (v4) at (1.5,1.5) {$\alpha_i^{(p+1)}$};
        \node (v5) at (3,0) {$\beta_i^{(p)}$};
        \node (v6) at (4.5,-1.5) {$\gamma_i^{(p-1)}$};
        \node (v8) at (6,0) {$\beta_{i+j}^{(p)}$};
        \node (v9) at (7.5,-1.5) {$\gamma_{i+j}^{(p-1)}$};
        \node (v3) at (0.7482,0.7688) {$\dots$};
        \node (v7) at (5.2021,-0.7445) {\dots};
        \draw  [-stealth](v1) edge (v2);
        \draw  [red][-stealth](v2) edge (v3);
        \draw  [red][-stealth](v3) edge (v4);
        \draw  [red][-stealth](v4) edge (v5);
        \draw  [red][-stealth](v5) edge (v6);
        \draw  [red][-stealth](v6) edge (v7);
        \draw  [red][-stealth](v7) edge (v8);
        \draw  [red][-stealth](v8) edge (v9);
        \end{tikzpicture}
         
        \caption{The figure shows a flowline through a simplicial complex, where the $\alpha$ simplices have dimension $p+1$, $\beta$ simplices have dimension $p$ and $\gamma$ simplices have dimension $p-1$. The $V$-path corresponding to the flowline is the path shown in red. Note that the downward arrows are via the boundary arrow, and the upward arrows are via the Morse function. The first and last simplices in the flowline are critical.}
        \label{fig:FlowTrajABC}
    \end{center}
\end{figure}
From this we can conclude that the first and last arrow must drop a dimension, given that critical simplices belong only to down-arrows.
\begin{definition}
    The \textbf{index} of a flow trajectory is the difference in dimension of the starting and concluding simplices. 
\end{definition}
Given that the dimension must decrease throughout the flow trajectory from critical $\alpha$ to critical $\gamma$, we may write
\[\Ind(\text{flowline}) = \dim(\alpha)-\dim(\gamma).\]
From here on, we use the term `flowline' to refer only to a flowline of index 2.

\begin{definition}
    For a gradient vector field $V$ on a simplicial complex with some flow trajectory $F:=\alpha_0^{(n+1)}\to\gamma_k^{(n-1)}$, the \textbf{intermediate simplex} $\beta_i$ is an $n$-simplex belonging to two pairs \[\{\alpha_i^{(n+1)},\beta_i^{(n)}\},\{\beta_i^{(n)},\gamma_i^{(n-1)}\}\in V\] with simplices of different dimension. 
\end{definition}
This is to say that if the flowline drops two dimensions, the intermediate simplex is the simplex at the point of the double drop. 

\begin{definition}
    A \textbf{critical flowline}\footnote{In differential geometry, critical flowlines are known more commonly as broken flowlines.} is a flowline with a critical simplex as an intermediate simplex. 
\end{definition}
We note that if the critical simplex $\beta$ is an intermediate simplex, then since it must have only down-arrows adjacent to it, the intermediate passage through $\beta$ must be a ``double drop point''. 
This is to say that the simplex before $\beta$ has one dimension more, and the simplex after has one dimension fewer.

\subsection{Orientation}
We define orientation on a simplex as an ordering on its vertices. On 1-simplices and 2-simplices, defining an orientation follows general intuition (see Figure \ref{fig:2simplexOrient}).
\begin{figure}[H]
    \centering
        \begin{tikzpicture}
        \node (v1) at (0,2) {1};
        \node (v2) at (-1.2,0.2) {2};
        \node (v3) at (1.2,0.2) {3};
        \draw  (v1) edge (v2);
        \draw  (v2) edge (v3);
        \draw  (v3) edge (v1);
        \draw  [red][-stealth]plot[smooth, tension=1] coordinates {(-0.2,0.8) (0,0.6) (0.2,0.8) (0.0379,0.9905)(-0.1457,0.9363)};
        \node (v4) at (-0.4,0.2) {};
        \node (v5) at (0.4,0.2) {};
        \node (v6) at (0.8,0.8) {};
        \node (v7) at (0.4,1.4) {};
        \node (v8) at (-0.4,1.4) {};
        \node (v9) at (-0.8,0.8) {};
        \draw  [red][-stealth](v4) edge (v5);
        \draw  [red][-stealth](v6) edge (v7);
        \draw  [red][-stealth](v8) edge (v9);
        \end{tikzpicture}
    \caption{The figure shows a 2-simplex endowed with an orientation, where the arrows on the edges are the induced orientation on the 1-simplices. 
    The orientation induced on the edge (12) induces some orientation on the vertex (1) which is opposite to the orientation that the edge (13) induces on the same vertex.
    }
    \label{fig:2simplexOrient}
\end{figure}
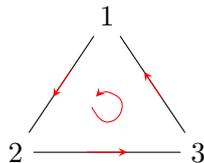
We can compare orientations between a simplex and its faces. Thus, each path in the Hasse diagram obtains a sign $(\pm 1)$.

We see from the figure that no matter what the orientation defined on the critical simplices, the induced orientation on some vertex $v$ via the edge $\beta_1$, say 1 without loss of generality, must be the opposite from the induced orientation on the vertex $v$ via $\beta_2$, say $(-1)$. We will prove that the paths through $\beta_1$ and $\beta_2$ have opposite signs. 
We find then that the action in Figure \ref{fig:FlopInHasse} taking the red path to the blue path flips the sign on the lower dimensional simplex in the case of 2-simplices. We wish to use induction to prove that this is the case for all simplices, but this is a bit tricky as there does not exist such an intuitive way for defining orientation on higher dimensional simplices as there does for triangles and edges.

For a simplex, we may define the `orientation' via a global ordering `$<$' placed on the vertices. For instance, on a 3-simplex we can endow the orientation given by the ordering $1>2>4>3$.
\begin{figure}[H]
    \centering
    \begin{tikzpicture}
    \node (v1) at (0,2) {1};
    \node (v2) at (-1.2,0.2) {2};
    \node (v3) at (0.8,0.1) {3};
    \draw  (v1) edge (v2);
    \draw  (v2) edge (v3);
    \draw  (v3) edge (v1);
    \node (v4) at (1.6,0.8) {4};
    \draw  (v3) edge (v4);
    \draw  [dashed](v2) edge (v4);
    \draw  (v1) edge (v4);
    \draw  [blue][-stealth]plot[smooth, tension=.7] coordinates {(-0.2593,1.8023) (-1.1255,0.5354)};
    \draw  [blue][-stealth]plot[smooth, tension=.7] coordinates {(-0.9962,0.348) (1.3436,0.8456)};
    \draw  [blue][stealth-]plot[smooth, tension=.7] coordinates {(1.0915,0.2058) (1.4341,0.5225)};
    \draw  [blue][-stealth]plot[smooth, tension=.7] coordinates {(0.5808,0.3027) (-0.0008,1.7312)};
    \end{tikzpicture}
    \caption{The figure shows a possible orientation on a tetrahedron.}
    \label{fig:tetrorient}
\end{figure}
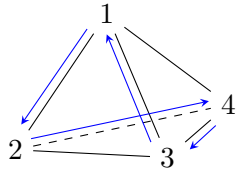
A way of defining orientation on a simplex of higher dimension is via an ordering of the vertices.
Let $\alpha$ be some $n$-simplex, with some orientation.
 
This orientation corresponds to defining an ordering `$>$' on the $n+1$ vertices in $\alpha$. 
For an arrow between simplices $\alpha^{(n)}$ and $\beta^{(n-1)}$ whose dimensions differ by one, we can use the orientation on each simplex to induce a sign on the arrow in a modified Hasse diagram between the simplices.
 
We assign to each arrow of a modified Hasse diagram an element of $\{\pm1\}$ by comparing orderings. 
For $\alpha$ an $n$-simplex with orientation defined by the order $>$, we find that by removing a vertex\footnote{When we remove a vertex from a simplex, we must also remove all simplices containing that vertex by implication, before considering it in a simplicial way.}, we obtain an $(n-1)$-simplex $\beta$. The orientation on $\alpha$ induces a sign on the arrow $\alpha\to\beta$.
\begin{definition}
    Fix an orientation for all simplices on $K$ by ordering $V$. Let $a\coloneq \alpha\to\beta$ be an arrow in the modified Hasse diagram of a simplicial complex. Assume $\alpha$ is defined by $[v_0,\dots,v_n]$, its ordered list of vertices. If $\beta=\alpha\backslash\{v_i\}$ then the sign $\theta(a)$ of the arrow $a$ of the Hasse diagram is $(-1)^i$.
\end{definition}

Using this definition of sign, we wish to define sign on a path.

\begin{definition}[p121, \cite{skoldberg}]
    Let $P$ be some path through a simplicial complex via edges in the modified Hasse diagram, and $\ell(P)$ be the number of arrows in path $P$. Then the sign of $P$ is given by
    \[\theta(P)\coloneq\left(\prod_{a\in P}\theta(a)\right)\left((-1)^{\frac{\ell(P)-\Ind(P)}{2}}\right).\]
\end{definition}
Using the definition of sign on a path, we can show the following identity.
\begin{lemma}\label{lem:multiplicative sign}
    For a simplicial complex $K$ with gradient vector field $V$, and two paths $P_1,P_2$ through $K$ whose composition $P_1\circ P_2$ is also a path through $K$, 
    \[\theta(P_1)\cdot\theta(P_2)=\theta(P_1\circ P_2)\]
    holds.
\end{lemma}
\begin{proof}
    For $P_1\circ P_2$ some path through $K$, we use the fact that $\Ind(P_1\circ P_2)=\Ind(P_1)+\Ind(P_2)$ to show that 
    \begin{align*}
        \left((-1)^{\frac{\ell(P_1\circ P_2)}{2}-\frac{\Ind(P_1\circ P_2)}{2}}\right)&=\left((-1)^{\frac{\ell(P_1)}{2}+\frac{\ell(P_2)}{2}-\frac{(\Ind(P_1)}{2}+\frac{\Ind(P_2)}{2}}\right)\\
        &=\left((-1)^{\frac{\ell(P_1)-\Ind(P_1)}{2}+\frac{\ell(P_2)-\Ind(P_2)}{2}}\right)\\
        &=\left((-1)^{\frac{\ell(P_1)-\Ind(P_1)}{2}}\right)+\left((-1)^{\frac{\ell(P_2)-\Ind(P_2)}{2}}\right)\\
        &=\theta(P_1)\cdot\theta(P_2),
    \end{align*}
    as required.
\end{proof}

We wish to use this to show that every pair of critical flowlines have opposite signs.
First, though, let us rigorously define the boundary function $\delt$.

\subsection{Setting boundaries}
We now wish to unify the concept of a boundary arrow with the idea of taking a simplex to one of its facets in an algebraic structure. 

\begin{definition}
    We say that the simplicial $i$-chains of $K$, denoted $C_i(K,\Z)$, is the free abelian group generated by the $i$-simplices of $K$.
\end{definition}
We can define a homology of a simplicial complex. For $\alpha,\beta$ simplices such that $\dim(\alpha)=\dim(\beta)+1$, let $\M(\alpha,\beta)$ be the set (or `0-dimensional manifold') of flowlines from simplex $\alpha$ to $\beta$ when $K$ is equipped with the canonical Morse function $f=(\sigma)=\dim(\sigma)$, and let $B$ be the set of critical $(i-1)$-simplices. Then the differentials on simplicial chains $\delt_i:C_i(K,\Z)\to C_{i-1}(K,\Z)$ are precisely the boundary arrows from the space of $i$-simplices to the space of $(i-1)$-simplices given by $\delt\left(\alpha^{(i)}\right)=\sum_{\beta\in B}\# \M(\alpha,\beta)\cdot \beta$. 
We prove that
\[\del_{i-1}\circ \del_i=0,\]
in Lemma \ref{lem:floppability assertion}.
Indeed, we can define homology groups as follows.

\begin{definition}[The $i$th Homology group]
    Let $K$ be a simplicial complex, where $K_i$ is the set of all $i$-dimensional simplices, and $\del_i$ is the boundary arrow $\del_i:C_i(K,\Z)\to C_{i-1}(K,\Z)$. 
    Then, the $i$th homology group is given by \[H_i\coloneq \frac{\ker(\del_i)}{\im(\del_{i+1})}.\]
\end{definition}

\section{The Morse geometric identity}\label{sec:The fundamental lemma of Morse theory}

In this section, we will define an algorithm by which we generate a one-dimensional manifold of index 2 flowlines through a simplicial complex equipped with a Morse function. We will show the algorithm acting on two examples. Finally, we see how the sign of a path is affected by the substeps of the algorithm. 

\begin{definition}
    Recall that $\M(\alpha,\beta)$ is the set of flowlines from simplex $\alpha$ to $\beta$. 
    Each point in $\M(\alpha,\beta)$ is a flowline, with a sign of $\pm 1$, from $\alpha$ to $\beta$, where $\dim(\alpha)=\dim(\beta)+1$. When the index is 1 we denote by $\#\M(\alpha,\beta)$ the signed count of such points: in other words, it is the number of flowlines with sign $(+1)$ minus the number of flowlines with sign $(-1)$.
    A Morse chain is a linear combination of critical simplices, where the linear combination of critical $p$-simplices is denoted $C_p^{\morse}(K,f)$. 
    The Morse differential, $\delt$, applied to a critical simplex $\alpha^{(p+1)}$ is defined by structure coefficients given by the size of the moduli space made of the flowlines between $\alpha$ and each critical simplex $\beta$ of dimension $p$. This is to say that $\delt\alpha$ is the sum of critical simplices $\beta$ of dimension $p$ each weighted by the signed count
    of flowlines $\alpha\to \beta$:
    \begin{equation}\label{eq:boundarymap}
        \delt\left(\alpha^{(p+1)}\right)=\sum_{\beta^{(p)} \text{ crit}}\# \M(\alpha,\beta)\cdot \beta.
    \end{equation}
\end{definition}

Under the standard inner product $\langle\cdot,\cdot\rangle$ on $C^{\morse}_{\bullet}(K,f)$ the space of Morse chains, we can also express equation \ref{eq:boundarymap} via the relation $\langle\delt \alpha,\beta\rangle\coloneq \#\M(\alpha,\beta)$.
The fundamental result of \cite{forman} is that $\delt$ gives $C^{\morse}_{\bullet}(K,f)$ the structure of a chain complex.
\begin{theorem}\label{thm:d2=0}
    The Morse differential, $\delt$, satisfies the relation that for all critical $\alpha^{(p+1)}, \gamma^{(p-1)}$ we have that $\langle\delt^2(\alpha),\gamma\rangle=0$ . 
\end{theorem}
This is to say, $\delt^2=0$. 
However, it feels vague and hard to conceptualise so far, so let us give a sketch of how we intend to prove this.
\begin{enumerate}[label = (\alph*)]
    \item We first show in Lemma \ref{lem:delt.is.orientation} that $\langle\delt(\alpha^{(p+1)}),\gamma^{(p-1)}\rangle$ is equivalent to taking the signed count of paths $\alpha\to\gamma$ through every possible critical $\beta$.
    \item We define an algorithm on index 2 flowlines in Definition \ref{def:algorithm} which terminates at a flowline which passes through a critical simplex. The algorithm defines an equivalence relation of flowlines, which we prove in \ref{lem:equiv.rel}. This, with the help of Lemma \ref{lem:algInvolutivity} gives us that the list of flowlines $\alpha^{(p+1)}\to\gamma^{(p-1)}$ generated by the algorithm endows $\M(\alpha,\gamma)$ with the structure of a 1-manifold, as shown in Lemma \ref{lem:2critFsPerEquivClass}.
    This tells us that every boundary flowline of $\M(\alpha,\gamma)$ has a unique and distinct partner flowline, where the algorithm is an involutive way of passing between the two.
    \item We then show that by taking only the critical flowlines, we can count these flowlines without identifying a critical intermediate simplex $\beta^{(p)}$ in each case, as proved in Lemma \ref{lem:beta.is.unimportant(no offence beta)}.
    We prove also that the boundary flowlines of $\M(\alpha,\gamma)$ correspond exactly with critical flowlines between $\alpha$ and $\gamma$.
    \item Following the definition of the algorithm, in Section \ref{subsec:florientations} we show that every substep - or `floperation' - of the algorithm negates the sign of the path. We then show in Lemma \ref{1mod2floperations} that there are always an odd number of floperations in the journey between one critical flowline and its unique and distinct partner, so that the sign of the unique distinct partner of critical $\widetilde{F}$ is the negative of the sign of $\widetilde(F)$.
\end{enumerate}
We expand on the following proof of Theorem \ref{thm:d2=0} throughout the rest of the paper.
\begin{proof}
    For $\alpha$ some critical $(n+1)$-simplex in a simplicial complex $K$ on which Morse function is defined, $B$ the set of critical $n$-simplices, $\Gamma$ the set of $(n-1)$-simplices, $\F_{\beta}(\alpha,\gamma)$ the set of flowlines with $\beta$ as an intermediate simplex, $\F(\alpha,\gamma)$ the set of all critical flowlines, and $\E(\alpha,\beta)$ the set of connected components of $\M(\alpha,\gamma)$,
    we have that

\begin{align*}
    \delt^2(\alpha)&=
        \sum_{\beta \in B}\sum_{\gamma \in \Gamma}\# \M(\alpha,\beta)\cdot\#\M(\beta,\gamma)\cdot \gamma \quad&&\text{by Lemma \ref{lem:floppability assertion}}\\
        &=\sum_{\gamma \in \Gamma}\sum_{\beta \in B}\sum_{F\in \F(\alpha,\gamma)_{\beta}}\theta(F)\cdot \gamma\quad &&\text{by observation}\\
        &=\sum_{\gamma \in \Gamma}\sum_{F\in \F(\alpha,\gamma)}\theta(F)\cdot \gamma\quad&&\text{by Lemma \ref{lem:del=boundarysum2}}\\
        &=\sum_{\gamma\in\Gamma}\sum_{e\in\E(\alpha,\beta)}\left(\theta(\Alg(F_i))\cdot\gamma+\theta(\Alg(\Alg(F_i)))\cdot\gamma\right)\quad&&\text{by Corollary \ref{cor:theta.alg=-theta.alg.alg}}
        \\
        &=\sum_{\gamma\in\Gamma}\sum_{e\in\E(\alpha,\beta)}\left(\theta(\Alg(F_i))\cdot\gamma-\theta(\Alg(F_i))\cdot\gamma\right)
        \\
        &=0.
\end{align*}
\end{proof}

Let us tackle the first step.
\begin{lemma}\label{lem:delt.is.orientation}
    Let $\alpha$ be a $(p+1)$-simplex, $B$ be the set of critical $p$-simplices and $\Gamma$ the set of $(p-1)$-simplices. Furthermore, for $\beta\in B$, let $\F_{\beta}(\alpha,\gamma)$ be the set of flowlines $\alpha\to\gamma$ with a double drop at $\beta$, $\F(\alpha,\gamma)$ be the set of critical flowlines $\alpha\to \gamma$, and $\theta$ be the sign on a flowline. Then using \ref{lem:multiplicative sign},
    we find that
    \[
        \sum_{\beta \in B}\sum_{\gamma \in \Gamma}\# \M(\alpha,\beta)\cdot\#\M(\beta,\gamma)\cdot \gamma=\sum_{\gamma \in \Gamma}\sum_{\beta \in B}\sum_{F\in \F_{\beta}(\alpha,\gamma)}\theta(F)\cdot \gamma\]
\end{lemma}

\begin{proof}
    To prove this claim it suffices to show that for given $\gamma\in \Gamma$ and $\beta\in B$, we have that 
    \[\# \M(\alpha,\beta)\cdot\#\M(\beta,\gamma)=\sum_{F\in \F_{\beta}(\alpha,\gamma)}\theta(F).\]

    Let $\F_{\beta}(\alpha,\gamma)$ have $l_1$ paths with sign $(+1)$ (i.e. paths $L_1$ with $\theta(L_1)=1$) and $l_2$ paths with sign $(-1)$ (i.e. paths $L_2$ with $\theta(L_2)=-1$). Then $\sum_{F\in \F_{\beta}(\alpha,\gamma)}\theta(F)=l_1-l_2$. This is the signed count of paths $\alpha\to\gamma$ through $\beta$.
    
    Indeed, $\#\M(\alpha,\beta)$ is the signed count of paths from $\alpha$ to $\beta$: say $\M(\alpha,\beta)$ has $n_1$ paths with sign $(+1)$ and $n_2$ paths with sign $(-1)$. Then $\#\M(\alpha,\beta)=\sum_{P\in\M(\alpha,\beta)} \theta(P)=n_1-n_2$. 
    Assuming also that $\M(\beta,\gamma)$ has $m_1$ paths with sign $(+1)$ and $m_2$ paths with sign $(-1)$, $\#\M(\beta,\gamma)=\sum_{P\in\M(\beta,\gamma)} \theta(P)=m_1-m_2$.

    For any path $P_1 \in \M(\alpha,\beta)$ and any $P_2\in\M(\alpha,\beta)$, the sign of the composite path $P_1\circ P_2$ is $\theta(P_1)\cdot\theta(P_2)$. 
    Suppose a path $P_1$ in $\M(\alpha,\beta)$ has sign $(+1)$. The signed count of paths from $\beta$ to $\gamma$ is $m_1-m_2$, so the signed count of all paths from $\alpha$ to $\gamma$ running first through path $P_1$ is $(+1)\cdot (m_1-m_2)$. Picking instead some $P_2\in\M(\alpha,\beta)$ with sign $(-1)$, the signed count of paths running first through $P_2$ is $(-1)\cdot(m_1-m_2)$.
    
    Since every path $F\in \F_{\beta}(\alpha,\gamma)$ is composed of some path $\alpha\to \beta$ and some path $\beta \to\gamma$, we see that the signed count of paths $\alpha\to\gamma$ through $\beta$ is $(n_1-n_2)\cdot (m_1-m_2)$, giving us the required inequality
    \[\# \M(\alpha,\beta)\cdot\#\M(\beta,\gamma)=(n_1-n_2)\cdot (m_1-m_2)=l_1-l_2=\sum_{F\in \F_{\beta}(\alpha,\gamma)}\theta(F).\]

    Considering that $B$ and $\Gamma$ are finite sets, we are perfectly justified in swapping the order of series, in the following way:
    \begin{align*}
        \sum_{\beta \in B}\sum_{\gamma \in \Gamma}\# \M(\alpha,\beta)\cdot\#\M(\beta,\gamma)\cdot \gamma
        &=\sum_{\beta \in B}\sum_{\gamma \in \Gamma}(n_1-n_2)\cdot (m_1-m_2)\cdot \gamma\\
        &=\sum_{\gamma \in \Gamma}\sum_{\beta \in B}(n_1-n_2)\cdot (m_1-m_2)\cdot \gamma\\
        &=\sum_{\gamma \in \Gamma}\sum_{\beta \in B}(l_1-l_2)\cdot \gamma\\
        &=\sum_{\gamma \in \Gamma}\sum_{\beta \in B}\sum_{F\in \F_{\beta}(\alpha,\gamma)}\theta(F)\cdot \gamma.
    \end{align*}
    Then the equality in the lemma holds.
\end{proof}
 
\subsection{The canonical Morse function}\label{subsec:CanonicalMorseFunction}
In this section, we show that $\delt^2=0$ for a simplicial complex on which the canonical Morse function is defined. This recovers Lemma \ref{lem:headtailorneither}.

In the case of the canonical Morse function $f$, we have that $f(\sigma)=\dim(\sigma)$ for all $\sigma\in K$, which is to say that every simplex is critical, as those with lower dimension have lower Morse value, and those with higher dimension have higher Morse value. In our first discussions of Hasse diagrams, all the arrows go downwards.  

We claim that $\delt^2(\alpha)=0$ for all $\alpha$ in a simplicial complex endowed with the canonical Morse function. Indeed, it is enough to show that for all $\alpha^{(n+1)},\gamma^{(n-1)}$ there are exactly 2 $n$-simplices $\beta_1,\beta_2$ such that both $\beta_1,\beta_2$ are facets of $\alpha$ and $\gamma$ is a facet of both $\beta_1$ and $\beta_2$, and that these two paths have opposite signs.
Equivalently, we wish to show that for each double drop in dimension 
\[\alpha^{(n+1)}\to \beta_1^{(n)}\to \gamma^{(n-1)}=:P_{1}\]
there is a unique $\beta_2^{(n)}$ such that $\beta_1\neq \beta_2$ and 
\[\alpha^{(n+1)}\to \beta_2^{(n)}\to \gamma^{(n-1)}=:P_2,\]
where $\theta(P_1)=-\theta(P_2)$.

\begin{lemma}[Floppability assertion]\label{lem:floppability assertion}
    Let $K$ be a simplicial complex. For some flowline $F$ through $C$ with a double drop $\alpha^{(p+1)}\to \beta_1^{(p)} \to \gamma^{(p-1)}$ there exists a unique $\beta_2\neq \beta_1^{(p)}$ such that $\beta_2$ is a facet of $\alpha$, and $\gamma$ is a facet of $\beta_2$. 
\end{lemma}
\begin{example}
    There is exactly one path from a $(p+1)$-simplex $\alpha$ to $\beta$, one of its facets. This is the boundary arrow.
    
    This is also true for a $p$-simplex $\beta$ with one of its $(p-1)$ dimensional facets, $\gamma$. We note that the $(p-1)$-simplex must be a face of the $(p+1)$-simplex. We show via Figure \ref{fig:2-simplex} that Lemma \ref{lem:floppability assertion} is true for a 3-simplex, by inducing on a 2-simplex, which builds up on the example of the 2-simplex we gave in the introduction. 

    \begin{figure}[H]
        \centering
        \begin{tikzpicture}
\node (v1) at (0,2) {$ v_0$};
\node (v2) at (-1.2,0.2) {$ v_1$};
\node (v3) at (0.8,0.1) {$ v_2$};
\draw  (v1) edge (v2);
\draw  (v2) edge (v3);
\draw  (v3) edge (v1);
\node at (-0.1,0.9) {$\alpha$};
\node at (0,0) {$\beta_1$};
\node at (0.7,1) {$\beta_2$};
\node (v4) at (1.6,0.8) {$ v_3$};
\draw  [red](v3) edge (v4);
\draw  [red][dashed](v2) edge (v4);
\draw  [red](v1) edge (v4);
\end{tikzpicture}
        \caption{The figure shows a 2 simplex with a 0-simplex adjoined. To see why we expect this to hold for the 3-simplex, we observe that we may add a 0-simplex, $ v_3$, and examine the simplicial complex of the tetrahedron $ v_0 v_1 v_2 v_3$ (we can imagine joining all 0-simplices to this point via 1-simplices and all 1-simplices to this point via 2-simplices). Here, we now consider the 1-simplex formed by connecting $ v_2$ to $ v_3$. 
        The 2-simplices formed by connecting $\beta_1$ and $\beta_2$ to $ v_3$ are therefore the only two facets of the tetrahedron ($\alpha$ adjoined to $ v_3$) of which the edge $ v_2  v_3$ is a facet. }
        \label{fig:2-simplex}
    \end{figure}
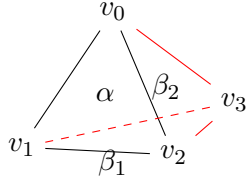
\end{example}
\begin{proof}    
    Consider an $(n+1)$-simplex $\alpha$, with some double drop $\alpha^{(n+1)}\to \beta_1^{(n)}\to \gamma^{(n-1)}$.
    
    We wish to show that there is a unique distinct $\beta_2$ such that both
    \[\alpha^{(n+1)}\to \beta_1^{(n)}\to \gamma^{(n-1)} \qand \alpha^{(n+1)}\to \beta_2^{(n)}\to \gamma^{(n-1)}\]
    are double drops.

    We acknowledge that, though all vertices of $\gamma$ are in $\alpha$, there are two vertices in $\alpha$ which are not in $\gamma$. In order to obtain $\gamma$ from $\alpha$ we may remove the vertices in two distinct ways, dependent only on the order in which we do so. This is to say that in having some path $\alpha\to \beta_1 \to \gamma$ there must be a unique alternative $\alpha\to \beta_2\to \gamma$.

    Here we have shown that the proposition holds when our simplicial complex $K$ is an $(n+1)$-simplex, thus proving that there are always exactly two paths down from some $n+1$ dimensional simplex to an $n-1$ dimensional simplex. 

    Furthermore, since any such flop in a simplicial complex must occur in some simplex, any double drop must have 2 paths down by the above logic.
    The logic therefore holds across simplicial complexes.
\end{proof}
The proof of the above fact is not enough to prove that the square of the Morse differential is trivial, but it does give us a useful tool for manipulating sequences of simplices. 
For $\alpha\to\beta\to\gamma$ some double drop in dimension, let us call the act of finding this unique $\beta'\neq\beta$ something memorable.
\begin{definition}
    For $\gamma$ a face of $\alpha$ where $\dim(\alpha)=\dim(\gamma)+2$, the \textbf{Flop} of the path $\alpha\to\beta\to\gamma$ is the unique and distinct path $\alpha\to\beta'\to\gamma$.
\end{definition}
We now show that the unique, distinct Flop of this double drop has the opposite sign to the double drop itself.

\begin{lemma}[The flip of the Flop]\label{lem:flip_of_flop}
    For some path $P$, we must have
    \[\theta(\Flop(P))=-\theta(P).\]
\end{lemma}

\begin{proof}
    Let there be some $(n+1)$-simplex, $\alpha$, which is to say that $\alpha$ has $n+2$ vertices. There are two vertices we must remove in order to reach an $(n-1)$-subsimplex of $\alpha$. Choose two such vertices $v_i$ and $v_j$, and let $\gamma\coloneq\alpha \backslash \{v_i,v_j\}$ be our $(n-1)$-simplex. Define an orientation on $\alpha$ via some ordering `$>$' on the vertices. Then there are two ways to obtain $\gamma$ from $\alpha$: first by removing $v_i$ then $v_j$, and second by removing $v_j$ then $v_i$. Suppose without loss of generality that $v_i$ is the $i$th vertex in the ordering `$>$', and $v_j$ is the $j$th. The act of removing $v_i$ gives the arrow between $\alpha$ and $\alpha\backslash\{v_i\}$ a sign of $(-1)^i$. Furthermore, in removing $v_j$ from $\alpha\backslash\{v_j\}$ the sign depends on whether $i<j$.
    \[\theta(\left\{\alpha\backslash\{v_i\}>\alpha\backslash\{v_i,v_j\}\right\})=\begin{cases}(-1)^i \quad &\text{if}\, i<j\\
    (-1)^{i+1} &\text{otherwise}
    \end{cases}.\]
    In removing $v_j$ first, we have the symmetric action, as shown in the diagram.
    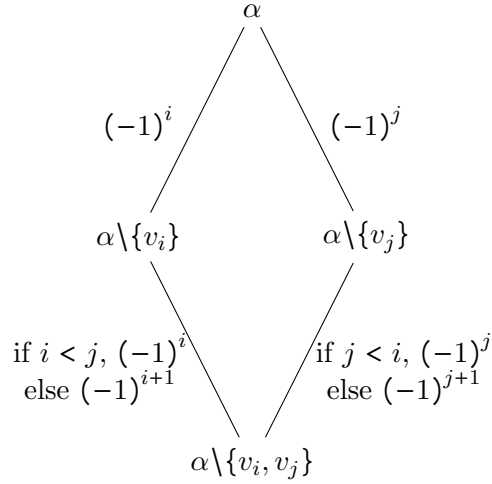
\begin{figure}[H]
        \centering
        \begin{tikzpicture}
        \node (v1) at (0.5,3) {$\alpha$};
        \node (v2) at (-1,0) {$\alpha\backslash \{v_i\}$};
        \node (v3) at (2,0) {$\alpha\backslash \{v_j\}$};
        \draw  (v1) edge (v2);
        \draw  (v3) edge (v1);
        \node (v4) at (0.5,-3) {$\alpha\backslash \{v_i,v_j\}$};
        \draw  (v2) edge (v4);
        \draw  (v4) edge (v3);
        \node at (2,1.5) {$(-1)^j$};
        \node at (-1,1.5) {$(-1)^i$};
        \node at (-1.5,-1.5) {if $i<j$, $(-1)^i$ };
        \node at (2.5,-1.5) {if $j<i$, $(-1)^j$ };
        \node at (-1.5,-2) {else $(-1)^{i+1}$ };
        \node at (2.5,-2) {else $(-1)^{j+1}$};
        \end{tikzpicture}
        \caption{The figure shows the paths from $\alpha$ to $\gamma$ via the removal of $v_i$ first and then that of $v_j$ first.}
        \label{fig:floppabilityOrient}
    \end{figure}
    This is to say that for $i<j$, if $P$ is the path removing $v_i$ then $v_j$, we have the sign $\theta(P)=1$, and $\theta(\Flop(P))=-1$, whereas for $i>j$ we have $\theta(P)=-1$ and $\theta(\Flop(P))=1$. I.e. the Flop operation negates the sign of the path, as required.    
\end{proof}
For any $(n+1)$-simplex $\alpha$ in a simplicial complex on which the canonical Morse function is defined, all critical $(n-1)$-simplices $\gamma$ have exactly two intermediate critical $(n)$-simplices, $\beta$ and $\beta'$, whose paths have opposite sign. We must therefore have 
\begin{align*}
    \delt^2(\alpha)&=
        \sum_{\gamma \in \Gamma}\sum_{\beta \in B}\sum_{F\in \F_{\beta}(\alpha,\gamma)}\theta(F)\cdot \gamma\\
        &=\sum_{\gamma \in \Gamma}\theta(\alpha\to\beta\to\gamma)\cdot \gamma+\theta(\alpha\to\beta'\to\gamma)\cdot \gamma\\
        &=\sum_{\gamma \in \Gamma}\theta(\alpha\to\beta\to\gamma)\cdot \gamma-\theta(\alpha\to\beta\to\gamma)\cdot \gamma\\
        &=\sum_{\gamma \in \Gamma}0=0.
\end{align*}

We conclude that any simplicial complex with the canonical Morse function value has $\delt^2(\alpha)=0$ using the definition of $\delt$ given. 
We wish now to show that $\delt^2(\alpha)=0$ regardless of the Morse function defined on the simplicial complex. For this, we need an analogue of the Flop for flowlines of length greater than 2. This, and more, will be defined in the next subsection.

\subsection{Algorithm}
Critical flow trajectories are like ghosts, scouring the Earth for their purpose before they can float away to the promised land. If a ghost finds his soulmate then they float off together and it is like they've never existed. In the same way, when a critical flow trajectory finds its partner, they cancel each other out and amount to zero. The search for such a partner is outlined in this subsection.

We use the Flop action to take an intermediate simplex to its unique partner simplex.
However, applying Flop twice just gets us back to the simplex we started with, so we will need a few more operations in our toolkit. So what are all the floperations we can use to manipulate a path of index 2?

\begin{itemize}
    \item \textbf{Flop}: This is the operation which switches the intermediate simplex for its unique distinct alternative. This is well defined, as there exists a unique double drop in any path of index 2, and the intermediate simplex $\beta$ has a unique, distinct alternative $\beta'$. 
    \begin{figure}[H]
        \centering
        \begin{tikzpicture}
\node (v2) at (-7,1) {$\alpha$};
\node (v1) at (-8,0) {\dots};
\node (v3) at (-6,0) {$\beta$};
\node (v4) at (-5,-1) {$\gamma$};
\node (v5) at (-4,0) {\dots};
\draw  [-stealth](v1) edge (v2);
\draw  [-stealth](v2) edge (v3);
\draw  [-stealth](v3) edge (v4);
\draw  [-stealth](v4) edge (v5);
\node (v6) at (-3,0) {};
\node (v7) at (-2,0) {};
\draw  [blue][dashed][-stealth](v6) edge (v7);
\node (v8) at (0,1) {$\alpha$};
\node (v9) at (-1,0) {\dots};
\node (v10) at (1,0) {$\beta'$};
\node (v11) at (2,-1) {$\gamma$};
\node (v12) at (3,0) {\dots};
\draw  [-stealth](v9) edge (v8);
\draw  [-stealth](v8) edge (v10);
\draw  [-stealth](v10) edge (v11);
\draw  [-stealth](v11) edge (v12);
\end{tikzpicture}
        \caption{The figure shows the Flop floperation, which takes the intermediate simplex (via the dashed blue arrow) to the unique distinct intermediate simplex connected to $\alpha$ and $\gamma$. Both of the flowlines here are legal. Not that this need not be the case.}
        \label{fig:howToFlop}
    \end{figure}
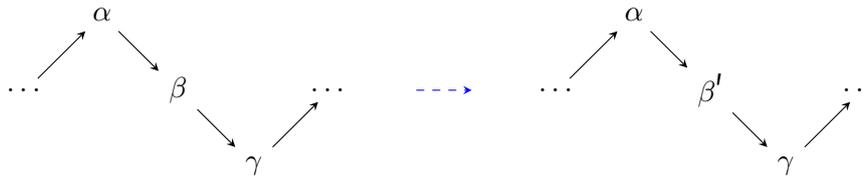
    \item \textbf{Insert}: We also have the ability with all noncritical flow trajectories (i.e. flowlines with a noncritical intermediate simplex) to Insert the Morse arrow attached - either immediately before or after - to the intermediate simplex. 
     
    This is well defined because we know there is a Morse arrow since the simplex is noncritical, and we know there can only be one by Lemma \ref{lem:headtailorneither}. We must add in two of these arrows to have a flowline that follows the narrative of increasing or decreasing dimension in each step to the right, and then we can perform the Flop function to rectify the fact that the double drop has an awkward backwards arrow. 
    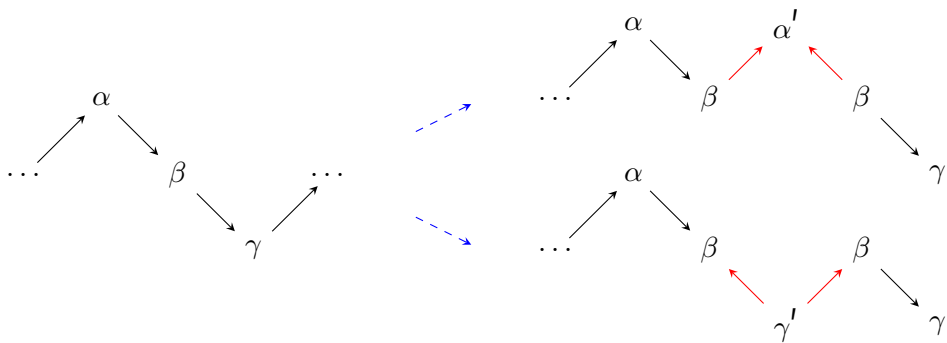
\begin{figure}[H]
        \centering
        \begin{tikzpicture}
\node (v2) at (-7,1) {$\alpha$};
\node (v1) at (-8,0) {\dots};
\node (v3) at (-6,0) {$\beta$};
\node (v4) at (-5,-1) {$\gamma$};
\node (v5) at (-4,0) {\dots};
\draw  [-stealth](v1) edge (v2);
\draw  [-stealth](v2) edge (v3);
\draw  [-stealth](v3) edge (v4);
\draw  [-stealth](v4) edge (v5);
\node (v6) at (-3,0.5) {};
\node (v7) at (-2,1) {};
\draw  [blue][dashed][-stealth](v6) edge (v7);
\node (v8) at (0,2) {$\alpha$};
\node (v9) at (-1,1) {\dots};
\node (v10) at (1,1) {$\beta$};
\node (v11) at (2,2) {$\alpha'$};
\node (v12) at (3,1) {$\beta$};
\draw  [-stealth](v9) edge (v8);
\draw  [-stealth](v8) edge (v10);
\draw  [red][-stealth](v10) edge (v11);
\draw  [red][stealth-](v11) edge (v12);
\node (v13) at (4,0) {$\gamma$};
\draw  [-stealth](v12) edge (v13);
\node (v14) at (-3,-0.5) {};
\node (v15) at (-2,-1) {};
\draw  [blue][dashed][-stealth](v14) edge (v15);
\node (v8) at (0,0) {$\alpha$};
\node (v9) at (-1,-1) {\dots};
\node (v10) at (1,-1) {$\beta$};
\node (v11) at (2,-2) {$\gamma'$};
\node (v12) at (3,-1) {$\beta$};
\draw  [-stealth](v9) edge (v8);
\draw  [-stealth](v8) edge (v10);
\draw [red] [stealth-](v10) edge (v11);
\draw [red] [-stealth](v11) edge (v12);
\node (v13) at (4,-2) {$\gamma$};
\draw  [-stealth](v12) edge (v13);
\end{tikzpicture}
        \caption{The figure shows the floperation of Insert, dependent on whether noncritical $\beta$ is paired in $V$ with a higher dimensional simplex (upper diagram) or with a lower dimensional simplex (lower diagram).}
        \label{fig:howToInsert}
    \end{figure}
    \item \textbf{Cancel}: Occasionally we might find ourselves with a double drop which has a backward arrow after flopping, in which case we can simply Cancel these out. To Cancel out a pair of simplices from the gradient vector field is to remove the secondary instance of the intermediate simplex and the simplex adjacent to both instances. This is well defined since there is at most one backwards arrow at any point (because all backwards arrows must be connected to the intermediate simplex, which in turn can only be adjacent to at most one Morse arrow) and it must be the same as the arrow either directly before or directly after it, so that the cancellation still gives a valid path.
    \begin{figure}[H]
        \centering
        \begin{tikzpicture}
\node (v2) at (8,1) {$\alpha$};
\node (v1) at (7,0) {\dots};
\node (v3) at (9,0) {$\beta$};
\node (v4) at (10,-1) {$\gamma$};
\node (v5) at (11,0) {\dots};
\draw  [-stealth](v1) edge (v2);
\draw  [-stealth](v2) edge (v3);
\draw  [-stealth](v3) edge (v4);
\draw  [-stealth](v4) edge (v5);
\node (v6) at (5,1) {};
\node (v7) at (6,0.5) {};
\draw  [blue][dashed][-stealth](v6) edge (v7);
\node (v8) at (0,2) {$\alpha$};
\node (v9) at (-1,1) {\dots};
\node (v10) at (1,1) {$\beta$};
\node (v11) at (2,2) {$\alpha'$};
\node (v12) at (3,1) {$\beta$};
\draw  [-stealth](v9) edge (v8);
\draw  [-stealth](v8) edge (v10);
\draw  [red][-stealth](v10) edge (v11);
\draw  [red][stealth-](v11) edge (v12);
\node (v13) at (4,0) {$\gamma$};
\draw  [-stealth](v12) edge (v13);
\node (v14) at (5,-1) {};
\node (v15) at (6,-0.5) {};
\draw  [blue][dashed][-stealth](v14) edge (v15);
\node (v8) at (0,0) {$\alpha$};
\node (v9) at (-1,-1) {\dots};
\node (v10) at (1,-1) {$\beta$};
\node (v11) at (2,-2) {$\gamma'$};
\node (v12) at (3,-1) {$\beta$};
\draw  [-stealth](v9) edge (v8);
\draw  [-stealth](v8) edge (v10);
\draw [red] [stealth-](v10) edge (v11);
\draw [red] [-stealth](v11) edge (v12);
\node (v13) at (4,-2) {$\gamma$};
\draw  [-stealth](v12) edge (v13);
\end{tikzpicture}
        \caption{The figure shows the floperation of Cancel, taking the simplex adjacent to the intermediate Morse arrow out completely.}
        \label{fig:howToCancel}
    \end{figure}
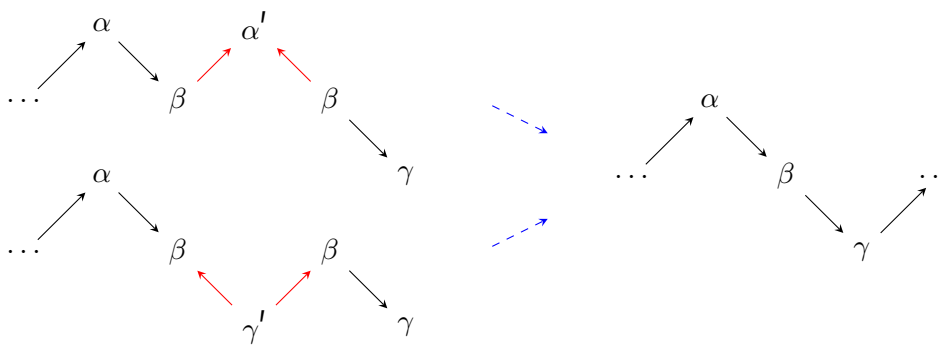
\end{itemize}
These three functions (Flop, Insert and Cancel) are the tools to build an algorithm to help us modify paths in the Hasse diagram,
and we will call them `floperations' of this algorithm. We can formalise the algorithm in the following way.

\begin{definition}
    We define an \textbf{illegal path} to be any undirected path with a backward arrow. \footnote{Note that an illegal path is not a flowline.} 
\end{definition}
We may describe the arrows adjacent to the intermediate simplex in an illegal path as a ``half-up half-down'' double drop since one of the arrows must be upwards facing (as the path is illegal) but one must be downwards facing (as no simplex can have two distinct Morse arrows).
Each step in the algorithm takes a legal flowline to the next legal flowline via some number of floperations, as described in the following lemma.
\begin{definition}\label{def:algorithm}
    To every critical flowline $\FO$ of index 2, we associate the label `c', and deterministically construct a list of flowlines $\F=[(\FO, \text{`c'}), (F_1, \ell_1), \dots (\Fn, \ell_n)]$ where $\FO$ and $\Fn$ are critical flowlines, each $\ell_i\in\{\text{`c', `f'}\}$, and each flowline in the list is given by the following algorithm. 
\begin{algorithmic}

\While{path is noncritical}
    \State Flop
    \While{path is illegal}
        \State Cancel
        \State Append to list with label `c'
        \State Flop
    \EndWhile
    \State Append to list with label `f'
    \If{flowline is critical}
        \State \textbf{break}
    \Else
        \State Insert
    \EndIf
\EndWhile
\end{algorithmic}
If you consider yourself more of a visual learner, Figure \ref{fig:algorithmFlowchart} is a flowline algorithm flowchart to help you along your way.
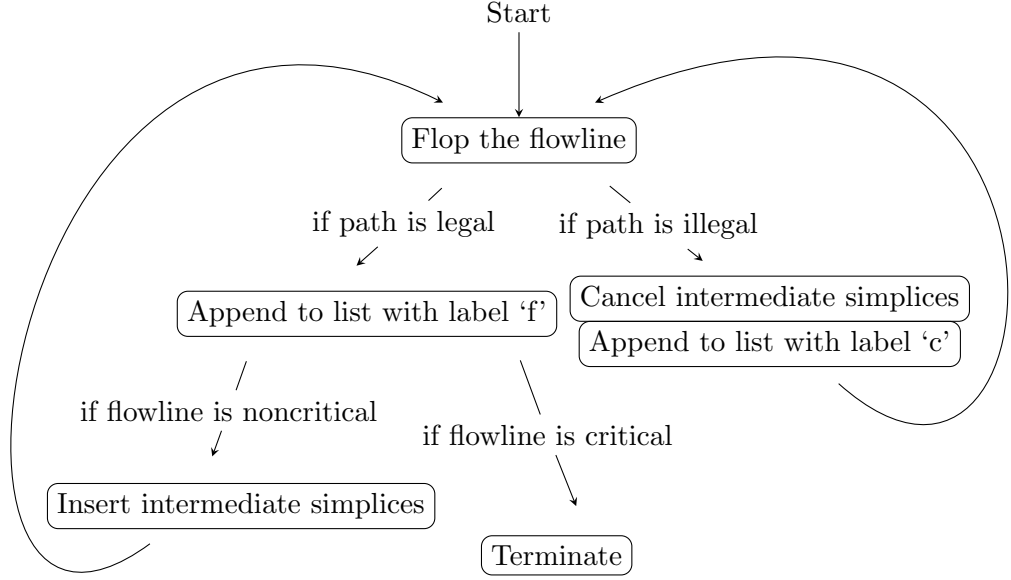
\begin{figure}[ht]
    \centering
\begin{tikzpicture}
\node[draw=black, rounded corners] (v14) at (-2.5,1.3) {Flop the flowline};
\node[draw=black, rounded corners] at (0.8,-0.8) {Cancel intermediate simplices};
\node[draw=black, rounded corners] at (0.8,-1.4) {Append to list with label `c'};
\node[draw=black, rounded corners] at (-4.5,-1) {Append to list with label `f'};
\node [draw=black, rounded corners]at (-6.1561,-3.549) {Insert intermediate simplices};
\node [draw=black, rounded corners]at (-2.0044,-4.2025) {Terminate};
\node (v9) at (-6.3175,-2.2818) {if flowline is noncritical};
\node (v1) at (-2.1212,-2.6038) {if flowline is critical};
\node (v3) at (-4.0157,0.1948) {if path is legal};
\node (v6) at (-0.6639,0.165) {if path is illegal};
\draw [-stealth](1.706,-1.9257) .. controls (4.9,-4.8) and (5.3,4.8) .. (-1.5,1.8)  ;
\draw [-stealth](-7.3561,-4.049) .. controls (-10.8,-6.6) and (-9.3,4.7) .. (-3.5,1.8)  ;
\node (v8) at (-6.0391,-1.4936) {};
\node (v10) at (-6.5849,-3.0262) {};
\node (v11) at (-2.5326,-1.4819) {};
\node (v12) at (-1.6939,-3.685) {};
\node (v2) at (-3.3801,0.7938) {};
\node (v4) at (-4.7735,-0.4834) {};
\node (v5) at (-1.4411,0.8054) {};
\node (v7) at (0.053,-0.4072) {};
\draw  (v2) edge (v3);
\draw  [-stealth](v3) edge (v4);
\draw  (v5) edge (v6);
\draw  [-stealth](v6) edge (v7);
\draw  (v8) edge (v9);
\draw  [-stealth](v9) edge (v10);
\draw  (v11) edge (v1);
\draw  [-stealth](v1) edge (v12);
\node (v13) at (-2.5,3) {Start};
\draw  [-stealth](v13) edge (v14);
\end{tikzpicture}
    \caption{The figure depicts the order of events in the algorithm, where we commence with the Flop operation and append each legal flowline with either `c' or `f' to account for just having Cancelled or Flopped.}
    \label{fig:algorithmFlowchart}
\end{figure}
 
\end{definition}
Notice that we can apply this algorithm to noncritical flowlines of index 2 as well, although it may not terminate. An example of this is shown in \S\ref{subsec:sphere_ex}. Furthermore, we can apply to a noncritical flowline a version of the algorithm which differs from the one given only by a shift of starting point. Since the flowline in question is noncritical, we can apply Insert as our first step and proceed with the algorithm as normal. 
\begin{notation}
    For some flowline $F$ we write $\Alg_c(F)$ for the critical flowline found by applying the algorithm to $F$ first by Flopping, and $\Alg_f(F)$ for the critical flowline found by applying the algorithm first by Inserting. We denote by $\AlgList_c(F)$ the list of flowlines generated by the algorithm where we Flop first (i.e. we have just Cancelled), and denote by $\AlgList_f$ the list generated by Inserting first (i.e. we have just Flopped).
\end{notation}

In the following discussions, it will be useful to think of the label c as dual to the label f, and in this way we shall denote the conjugate label $\overline{c}=f$, and conversely $\overline{f}=c$.

Here, the `Append' operation is only for bookkeeping, and does not alter the algorithm's final critical output. Append occurs at every legal flowline (which is true because it occurs after every instance of Cancel and before every instance of Insert) to register it in the list of flowlines. 
We claim that the algorithm will be involutive, in that it will take a critical flowline $F_{i}$ to another critical flowline $F_{j}$ and terminate, and will take $F_{j}$ to $F_{i}$ and terminate. 
This is to say that the terminating flowlines of the algorithm correspond exactly to critical flow trajectories.
The flowlines between $F_{i}$ and $F_{j}$ given in $\AlgList_f(F_{i})$ represent a deformation of $F_{i}$ to $F_{j}$. Indeed, one way to interpret our results is that the space of flowlines in such a list is a one-dimensional simplicial manifold with boundary. If $F_i,F_j$ start with simplex $\alpha^{(n+1)}$ and finish with simplex $\gamma^{(n-1)}$, we call this space of flowlines $\M(\alpha,\gamma)$, the moduli space of flowlines.
\begin{definition}(Flowline equivalence)\label{def:flowlineEquivalnece}
    For any two flowlines $F_{i}, F_{j}$ of index 2 we define the relation $\sim$ such that $F_{i}\sim F_{j}$ if there exists some $F'$ such that $\AlgList_{f/c}(F)=[\cdots, F_i, \cdots, F_j , \cdots]$.
\end{definition}
We will show in Lemma \ref{lem:equiv.rel} that this is indeed an equivalence relation. We see that, to act on some noncritical $F\in \AlgList_c(\FO)$ we have two options: $\AlgList_c(F)$ or $\AlgList_f(F)$, which each terminate at different critical flowlines. In this way, the equivalence class is a one-dimensional space of flowlines with two endpoints. We call the endpoints boundaries.
\begin{definition}
    A \textbf{boundary flowline} is a flowline at which the flowline algorithm terminates. The set of all boundary flowlines from $\alpha^{(n+1)}$ to $\gamma^{(n-1)}$ is denoted $\del\M(\alpha,\gamma)$.
\end{definition}

\begin{definition}\label{def:ModuliSpace}
    The moduli space $\M(\alpha,\gamma)$ of index 2 flowlines between two critical simplices $\alpha,\gamma$ with $\dim(\alpha)=\dim(\gamma)+2$ is the simplicial complex whose vertices are flowlines $F: \alpha\to \gamma$, and whose edges $\{F_i,F_j\}$ are pairs of subsequent flowlines in $\AlgList_f(F)$ for some choice of $F$.
\end{definition}

We will prove in Lemma \ref{lem:2critFsPerEquivClass} $\M(\alpha,\gamma)$ is a 1-dimensional simplicial manifold, and subsequently that the boundary of this manifold is exactly the boundary flowlines. 
 
We now show the following Lemma.
\begin{lemma}\label{lem:beta.is.unimportant(no offence beta)}
    For $\alpha$ a critical $(p+1)$-simplex, $B$ the set of critical $p$-simplices, $\Gamma$ the set of critical $(p-1)$-simplices, $\F_{\beta}(\alpha,\gamma)$ the set of flowlines through $\beta$ and $\F(\alpha,\gamma)$ the set of critical flowlines,
    \begin{align*}
        \sum_{F\in\F(\alpha,\gamma)}\theta(F)=\#\del \M(\alpha,\gamma).
    \end{align*}
\end{lemma}
To prove the lemma, we need only show that for some critical $(p+1)$-simplex $\alpha$, and $\Gamma$ the set of critical $(p-1)$-simplices, the following equality holds.
 
\begin{align*}
&\sum_{\gamma \in \Gamma}\sum_{\beta \in B}\sum_{F\in \F_{\beta}(\alpha,\gamma)}\theta(F)\\
&= \sum_{\gamma\in\Gamma}\#(\text{boundary flowlines } \alpha\to\gamma).
\end{align*}

\begin{proof}
    We claim that the boundary flowlines coincide exactly with critical flowlines. 
    All boundary flowlines must be critical by the construction of the algorithm, so we need only show that all critical flowlines are terminal flowlines of the algorithm. For the converse, we need to reach into the future and grab Lemma \ref{lem:equiv.rel}, which tells us that the flowline equivalence we defined in Definition \ref{def:flowlineEquivalnece} is an equivalence relation.
    We can then assume that every flowline $\alpha\to \gamma$ belongs to a partition of flowline equivalence, inferring that each critical flowline belongs to such a partition. Furthermore, by definition of the algorithm, any given critical flowline can only occur at the boundary of a partition. This is to say that all critical flowlines must be boundary flowlines.
    Then, since for each $\gamma$ we have 
    \begin{align*}
        \#(\text{flowlines } \alpha\to\gamma \text{ with 2 drops at some critical simplex})&\\=\#(\text{boundaries of flowlines } \alpha\to\gamma \text{ with 2 drops})&,
    \end{align*}
    the lemma holds.
\end{proof}

We can conclude that \[\delt^2(\alpha)=\sum_{\gamma\in\Gamma}\#(\text{boundaries of flowlines } \alpha\to\gamma \text{ with 2 drops})\cdot \gamma.\]

This conclusion leads us to try to find a way to pair up flowlines in this sum. 
In defining $\M(\alpha,\beta)$ for simplices $\alpha^{(n)},\beta^{(n-1)}$ as the space of flowlines from $\alpha$ to $\beta$, each point in $\M(\alpha,\beta)$ must be such a path, and each step of the algorithm must represent an edge connecting two such paths, as indicated in Definition \ref{def:ModuliSpace}. Then an equivalence class of flowlines is a connected component of $\M(\alpha,\beta)$.
Indeed, we claim that we can partition the sum into pairs of critical flowlines belonging to the same connected component of $\M(\alpha,\beta)$, as we will prove in Lemma \ref{lem:del=boundarysum2}.

\subsection{Sphere example}\label{subsec:sphere_ex}
Let us show the algorithm in action, first with a trivial example: that where there are no critical flow trajectories. Let us take a simplicial complex with a critical $2$-simplex called `123' and a critical 0-simplex called `4' as shown in the Hasse diagram, with the Morse function endowed. This simplicial complex represents the structure of a sphere and the critical simplices consist only of a critical maximum and a critical minimum. Note that there is a down arrow from a $p$-simplex to a $(p-1)$-simplex if its vertices are a subset of the $p$-simplex. For an arbitrary flowline $123\rightsquigarrow 4$, we can still apply the algorithm, but since the flowline must be noncritical it simply finds in which partition of the flow trajectory space the flowline lies. 
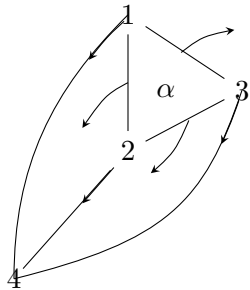
\begin{figure}[H]
    \centering
\begin{tikzpicture}

\node (v1) at (0,0) {1};
\node (v2) at (0,-1.8) {2};
\node (v3) at (1.5,-1) {3};
\node (v4) at (-1.5,-3.5) {4};
\draw  (v1) edge (v2);
\draw  (v2) edge (v3);
\draw  (v3) edge (v1);
\draw (1.5,-1) .. controls (1,-2.5) and (0.5,-3) .. (-1.5,-3.5);
\draw (0,0) .. controls (-1,-1) and (-1.5,-2) .. (-1.5,-3.5) node (v5) {};
\draw  (v5) edge (v2);
\node at (0.5,-1) {$\alpha$};
\draw [-stealth] plot[smooth, tension=.7] coordinates {(0,-0.9) (-0.3,-1.1) (-0.6,-1.5)};
\draw  [-stealth]plot[smooth, tension=.7] coordinates {(0.8,-1.4) (0.6,-1.8) (0.3,-2.1)};
\draw [-stealth] plot[smooth, tension=.7] coordinates {(0.7,-0.5) (1,-0.3) (1.4,-0.2)};
\draw  [-stealth]plot[smooth, tension=.7] coordinates {(-0.1861,-2.0246) (-0.6,-2.5)};
\draw [-stealth] plot[smooth, tension=.7] coordinates {(-0.0517,-0.086) (-0.2792,-0.3065) (-0.5135,-0.596)};
\draw [-stealth] plot[smooth, tension=.7] coordinates {(1.4576,-1.1404) (1.3542,-1.423) (1.2301,-1.7125)};
\end{tikzpicture}

    \caption{The figure shows a 2-simplex with a critical face, $\alpha$, and a critical vertex, $4$.}
    \label{fig:tetraSphere}
\end{figure}

The Hasse diagram is given by:

\begin{figure}[H]
    \centering
    \begin{tikzpicture}
\node (v14) at (-3,1.5) {123};
\node (v3) at (-1.5,1.5) {124};
\node (v5) at (0,1.5) {134};
\node (v7) at (1.5,1.5) {234};
\node (v9) at (0,0) {14};
\node (v6) at (-1.5,0) {23};
\node (v4) at (-3,0) {13};
\node (v2) at (-4.5,0) {12};
\node (v11) at (1.5,0) {24};
\node (v13) at (3,0) {34};
\node (v8) at (-3,-1.5) {1};
\node (v10) at (-1.5,-1.5) {2};
\node (v12) at (0,-1.5) {3};
\node (v1) at (1.5,-1.5) {4};
\draw  (-3,1.5) node (v14) {} ellipse (0.5 and 0.5);
\draw  (v1) ellipse (0.5 and 0.5);
\draw  [red][-stealth](v2) edge (v3);
\draw  [red][-stealth](v4) edge (v5);
\draw  [red][-stealth](v6) edge (v7);
\draw  [red][-stealth](v8) edge (v9);
\draw  [red][-stealth](v10) edge (v11);
\draw  [red][-stealth](v12) edge (v13);
\draw  [-stealth](v14) edge (v2);
\draw  [-stealth](v14) edge (v4);
\draw  [-stealth](v3) edge (v9);
\draw  [-stealth](v5) edge (v9);
\draw  [-stealth](v3) edge (v11);
\draw  [-stealth](v5) edge (v13);
\draw  [-stealth](v7) edge (v11);
\draw  [-stealth](v7) edge (v13);
\draw  [-stealth](v14) edge (v6);
\draw  [-stealth](v2) edge (v8);
\draw  [-stealth](v2) edge (v10);
\draw  [-stealth](v6) edge (v10);
\draw  [-stealth](v4) edge (v12);
\draw  [-stealth](v6) edge (v12);
\draw  [-stealth](v9) edge (v1);
\draw  [-stealth](v11) edge (v1);
\draw  [-stealth](v13) edge (v1);
\draw  [-stealth](v4) edge (v8);
\end{tikzpicture}
    \caption{The figure shows the Hasse diagram of the given 2-simplex, with Morse arrows given in red.}
    \label{fig:HasseTetraSphere}
\end{figure}
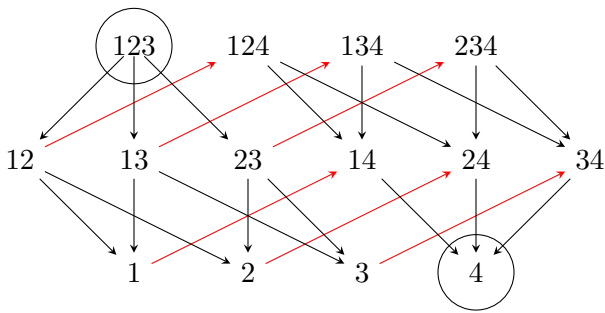

We select a flowline like this one.

\begin{center} \begin{tikzpicture}
\node (v1) at (0,4) {$123$};
\node (v2) at (1,3) {$12$};
\node (v3) at (2,4) {124};
\node (v4) at (3,3) {14};
\node (v5) at (4,2) {4};
\draw  [-stealth](v1) edge (v2);
\draw  [-stealth](v2) edge (v3);
\draw  [-stealth](v3) edge (v4);
\draw  [-stealth](v4) edge (v5);
\node at (-1.5,3) {$F_0:=$};
\end{tikzpicture}\end{center} 

Since this is not a critical flow trajectory, we start by Inserting the Morse arrow adjacent to 14. 

\begin{center}
\begin{tikzpicture}
\node (v123) at (0,4) {$123$};
\node (v12) at (1,3) {$12$};
\node (v124) at (2,4) {124};
\node (v14) at (3,3) {14};
\node (v1) at (4, 2) {1};
\draw  [-stealth](v123) edge (v12);
\draw  [-stealth](v12) edge (v124);
\draw  [-stealth](v124) edge (v14);
\draw  [red][stealth-](v14) edge (v1);
\node (v14) at (5,3) {14};
\node (v4) at (6,2) {4};
\draw [red][-stealth](v1) edge (v14);
\draw [-stealth](v14) edge (v4);
\end{tikzpicture}
\end{center}

We Flop to obtain the following.
\begin{center}
\begin{tikzpicture}
\node (v123) at (0,4) {$123$};
\node (v12) at (1,3) {$12$};
\node (v124) at (2,4) {124};
\node (v14) at (3,3) {12};
\node (v1) at (4,2) {1};
\draw  [-stealth](v123) edge (v12);
\draw  [-stealth](v12) edge (v124);
\draw  [red][stealth-](v124) edge (v14);
\draw  [red][-stealth](v14) edge (v1);
\node (v14) at (5,3) {14};
\node (v4) at (6,2) {4};
\draw [-stealth](v1) edge (v14);
\draw [-stealth](v14) edge (v4);
 
\end{tikzpicture}
\end{center}

We see that we can Cancel, so we do, to obtain the following legal path.
\begin{center}
\begin{tikzpicture}
\node (v123) at (0,4) {123};
\node (v12) at (1,3) {12};
\node (v1) at (2,2) {1};
\draw [red][-stealth](v123) edge (v12);
\draw [red][-stealth](v12) edge (v1);
\node (v14) at (3,3) {14};
\node (v4) at (4,2) {4};
\draw [-stealth](v1) edge (v14);
\draw [-stealth](v14) edge (v4);
\node at (-1.5,3) {$F_1:=$};
\end{tikzpicture}
\end{center}

Since we have Cancelled, we then apply Flop. This path too is legal.
\begin{center}
\begin{tikzpicture}
\node (v123) at (0,4) {123};
\node (v13) at (1,3) {13};
\node (v1) at (2,2) {1};
\draw  [red][-stealth](v123) edge (v13);
\draw  [red][-stealth](v13) edge (v1);
\node (v14) at (3,3) {14};
\node (v4) at (4,2) {4};
\draw [-stealth](v1) edge (v14);
\draw [-stealth](v14) edge (v4);
\node at (-1.5,3) {$F_2:=$};
\end{tikzpicture}
\end{center}

Continuing in this way, we find ourselves back at the original flowline after 11 Flopsworth of applying the algorithm. This is to say that the algorithm has no boundary flow trajectories, such that the algorithm never terminates. We can imagine $\M(123, 4)$ is in the form of a circle with some number of notches corresponding to flowlines. After a finite number of floperations, we cycle back to the same flowline again. This agrees with what we might expect given that there are no intermediate critical simplices.

\subsection{Projective plane example}
For another example, take the real projective plane $\RPtwo$. Here I will use an example from Forman's Article (Section 4, \cite{forman}) on which to apply the algorithm. The real projective plane is equivalent to a simplicial complex with a critical 0-, 1- and 2-simplex. 

There exists a critical flow through the simplicial complex in the following way.
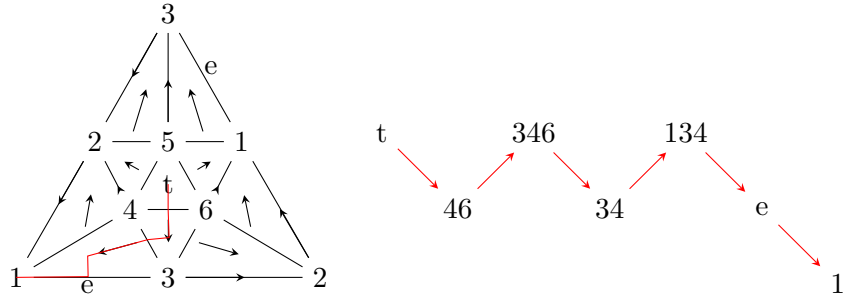
\begin{figure}[H]
    \centering
\begin{tikzpicture}
\node (v2) at (2,0.5) {3};
\node (v1) at (0,-3) {1};
\node (v3) at (4,-3) {2};
\node (v4) at (1.035,-1.2) {2};
\node (v6) at (2.965,-1.2) {1};
\node (v5) at (2,-3) {3};
\node (v7) at (2,-1.2) {5};
\node (v8) at (1.5,-2.1) {4};
\node (v9) at (2.5,-2.1) {6};
\draw  (v7) edge (v2);
\draw  (v1) edge (v8);
\draw  (v9) edge (v3);
\draw  (v8) edge (v9);
\draw  (v9) edge (v7);
\draw  (v7) edge (v8);
\draw  (v4) edge (v7);
\draw  (v7) edge (v6);
\draw  (v6) edge (v9);
\draw  (v5) edge (v9);
\draw  (v8) edge (v5);
\draw  (v4) edge (v8);
\draw  (v2) edge (v4);
\draw  (v4) edge (v1);
\draw  (v1) edge (v5);
\draw  (v5) edge (v3);
\draw  (v3) edge (v6);
\draw  (v6) edge (v2);
\node (v21) at (1.4517,-0.4807) {};
\node (v20) at (0.4822,-2.1672) {};
\node (v126) at (3.3917,-1.9586) {};
\node (v18) at (0.8867,-2.464) {};
\node (v19) at (1.008,-1.7639) {};
\node (v16) at (3.1133,-2.464) {};
\node (v17) at (2.973,-1.774) {};
\node (v30) at (1.75,-1.65) {};
\node (v28) at (2.25,-1.65) {};
\node (v31) at (1.2893,-1.3894) {};
\node (v29) at (2.6967,-1.3844) {};
\node (v32) at (1.265,-1.65) {};
\node (v26) at (2.735,-1.65) {};
\node (v10) at (1.73,-2.5) {};
\node (v14) at (2.27,-2.5) {};
\node (v11) at (0.9415,-2.7169) {};
\node (v15) at (3.0647,-2.732) {};
\node (v13) at (2.0056,-2.6706) {};
\node (v12) at (2,-2.1) {};
\node (v22) at (1.5,-1.2) {};
\node (v24) at (2.5,-1.2) {};
\node (v23) at (1.7525,-0.37) {};
\node (v25) at (2.2436,-0.365) {};
\node (v27) at (3.1416,-3.0004) {};
\draw  [-stealth](v10) edge (v11);
\draw  [-stealth](v12) edge (v13);
\draw  [-stealth](v14) edge (v15);
\draw  [-stealth](v16) edge (v17);
\draw  [-stealth](v18) edge (v19);
\draw  [-stealth](v4) edge (v20);
\draw  [-stealth](v2) edge (v21);
\draw  [-stealth](v22) edge (v23);
\draw  [-stealth](v24) edge (v25);
\draw  [-stealth](v9) edge (v26);
\draw  [-stealth](v8) edge (v32);
\draw  [-stealth](v30) edge (v31);
\draw  [-stealth](v28) edge (v29);
\draw  [-stealth](v5) edge (v27);
\draw  [-stealth](v3) edge (v126);
\node (v111) at (2.0001,-0.2445) {};
\draw  [-stealth] (v7) edge (v111);
\node at (2.5657,-0.2259) {e};
\node at (0.9373,-3.1258) {e};
\node (v33) at (2.00,-1.7638) {t};
\draw  [red]plot[smooth, tension=0] coordinates {(v33) (v12) (2.0128,-2.4725) (v10) (v11) (0.9474,-2.9902) (v1)};
\end{tikzpicture}
~
\begin{tikzpicture}
\node (t) at (0,4) {t};
\node (v46) at (1,3) {46};
\node (v346) at (2,4) {346};
\draw [red][-stealth](t) edge (v46);
\draw [red][-stealth](v46) edge (v346);
\node (v34) at (3,3) {34};
\node (v134) at (4,4) {134};
\draw [red][-stealth](v346) edge (v34);
\draw [red][-stealth](v34) edge (v134);
\node (e) at (5,3) {e};
\draw [red] [-stealth](v134) edge (e);
\node (v1) at (6,2) {1};
\draw [red][-stealth] (e) edge (v1);
\end{tikzpicture}

    \caption{The figure shows the path through the simplicial complex.}
    \label{fig:RP2Flowline1}
\end{figure}

We can see how this might represent the real projective plane. We know that the real projective plane can be described as $\RPtwo=S^2/\sim$ where $(x,y,z)\sim(-x,-y,-z)$, or even as $\RPtwo=(S^2,z\geq 0)/\sim$ where $(x,y,0)\sim(-x,-y,0)$, so we may take the upper hemisphere and identify the opposite ends of the equator. We see a similar thing happening in the simplicial complex, where the north pole, as the maximum, is given by $t$, and all the vertices on the equator, i.e. 1, 2 and 3, are identified with their opposite.  In the flowline, we are travelling down from the north pole to the critical point at the equator. 
Applying the algorithm to this flowline, we find that it terminates at the following blue flowline. 

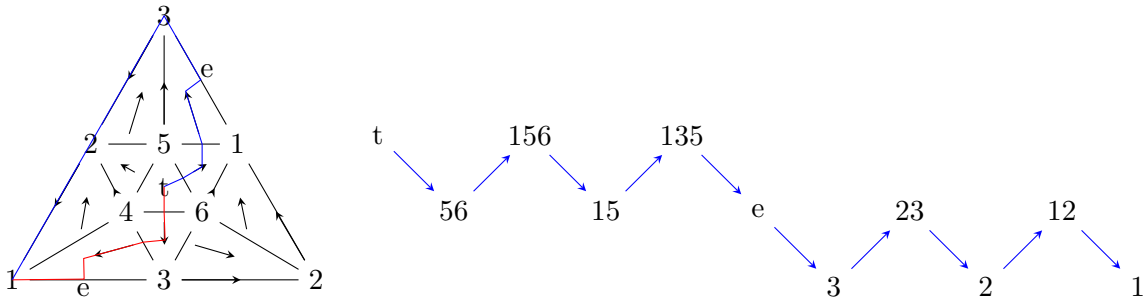
\begin{figure}[H]
    \centering
\begin{tikzpicture}
\node (v2) at (2,0.5) {3};
\node (v1) at (0,-3) {1};
\node (v3) at (4,-3) {2};
\node (v4) at (1.035,-1.2) {2};
\node (v6) at (2.965,-1.2) {1};
\node (v5) at (2,-3) {3};
\node (v7) at (2,-1.2) {5};
\node (v8) at (1.5,-2.1) {4};
\node (v9) at (2.5,-2.1) {6};
\draw  (v7) edge (v2);
\draw  (v1) edge (v8);
\draw  (v9) edge (v3);
\draw  (v8) edge (v9);
\draw  (v9) edge (v7);
\draw  (v7) edge (v8);
\draw  (v4) edge (v7);
\draw  (v7) edge (v6);
\draw  (v6) edge (v9);
\draw  (v5) edge (v9);
\draw  (v8) edge (v5);
\draw  (v4) edge (v8);
\draw  (v2) edge (v4);
\draw  (v4) edge (v1);
\draw  (v1) edge (v5);
\draw  (v5) edge (v3);
\draw  (v3) edge (v6);
\draw  (v6) edge (v2);
\node (v21) at (1.4517,-0.4807) {};
\node (v20) at (0.4822,-2.1672) {};
\node (v126) at (3.3917,-1.9586) {};
\node (v18) at (0.8867,-2.464) {};
\node (v19) at (1.008,-1.7639) {};
\node (v16) at (3.1133,-2.464) {};
\node (v17) at (2.973,-1.774) {};
\node (v30) at (1.75,-1.65) {};
\node (v28) at (2.25,-1.65) {};
\node (v31) at (1.2893,-1.3894) {};
\node (v29) at (2.6967,-1.3844) {};
\node (v32) at (1.265,-1.65) {};
\node (v26) at (2.735,-1.65) {};
\node (v10) at (1.73,-2.5) {};
\node (v14) at (2.27,-2.5) {};
\node (v11) at (0.9415,-2.7169) {};
\node (v15) at (3.0647,-2.732) {};
\node (v13) at (2.0056,-2.6706) {};
\node (v12) at (2,-2.1) {};
\node (v22) at (1.5,-1.2) {};
\node (v24) at (2.5,-1.2) {};
\node (v23) at (1.7525,-0.37) {};
\node (v25) at (2.2436,-0.365) {};
\node (v27) at (3.1416,-3.0004) {};
\draw  [-stealth](v10) edge (v11);
\draw  [-stealth](v12) edge (v13);
\draw  [-stealth](v14) edge (v15);
\draw  [-stealth](v16) edge (v17);
\draw  [-stealth](v18) edge (v19);
\draw  [-stealth](v4) edge (v20);
\draw  [-stealth](v2) edge (v21);
\draw  [-stealth](v22) edge (v23);
\draw  [-stealth](v24) edge (v25);
\draw  [-stealth](v9) edge (v26);
\draw  [-stealth](v8) edge (v32);
\draw  [-stealth](v30) edge (v31);
\draw  [-stealth](v28) edge (v29);
\draw  [-stealth](v5) edge (v27);
\draw  [-stealth](v3) edge (v126);
\node (v111) at (2.0001,-0.2445) {};
\draw  [-stealth] (v7) edge (v111);
\node at (2.5657,-0.2259) {e};
\node at (0.9373,-3.1258) {e};
\node (v33) at (2.00,-1.7638) {t};
\draw  [red]plot[smooth, tension=0] coordinates {(v33) (v12) (2.0128,-2.4725) (v10) (v11) (0.9474,-2.9902) (v1)};
\draw  [blue]plot[smooth, tension=0] coordinates {(v33) (v28) (2.5,-1.5) (v24) (2.2826,-0.4986) (2.4829,-0.3479) (v2) (v1)};
\end{tikzpicture}
~
\begin{tikzpicture}
\node (t) at (0,4) {t};
\node (v46) at (1,3) {56};
\node (v346) at (2,4) {156};
\draw [blue][-stealth](t) edge (v46);
\draw [blue][-stealth](v46) edge (v346);
\node (v34) at (3,3) {15};
\node (v134) at (4,4) {135};
\draw [blue][-stealth](v346) edge (v34);
\draw [blue][-stealth](v34) edge (v134);
\node (e) at (5,3) {e};
\node (v3) at (6,2) {3};
\draw  [blue][-stealth](v134) edge (e);
\draw  [blue][-stealth](e) edge (v3);
\node (v23) at (7,3) {23};
\node (v2) at (8,2) {2};
\node (v12) at (9,3) {12};
\node (v1) at (10,2) {1};
\draw  [blue][-stealth](v3) edge (v23);
\draw  [blue][-stealth](v23) edge (v2);
\draw  [blue][-stealth](v2) edge (v12);
\draw  [blue][-stealth](v12) edge (v1);
\end{tikzpicture}
    \caption{The figure shows the other critical flowline through the simplicial complex.}
    \label{fig:RP2Flowline2}
\end{figure}
We can split each path into two components: one giving the half flowline $t\to e$, and the other giving the half flowline $e\to 1$. The possible flowlines $t\to e$ are the following.

\begin{center}
\begin{tikzpicture}
\node (t) at (-1,6.5) {t};
\node (v46) at (0,5.5) {46};
\node (v346) at (1,6.5) {346};
\draw [red][-stealth](t) edge (v46);
\draw [red][-stealth](v46) edge (v346);
\node (v34) at (2,5.5) {34};
\node (v134) at (3,6.5) {134};
\draw [red][-stealth](v346) edge (v34);
\draw [red][-stealth](v34) edge (v134);
\node (e) at (4,5.5) {e};
\draw [red] [-stealth](v134) edge (e);

\node (t) at (6,6.5) {t};
\node (v46) at (7,5.5) {56};
\node (v346) at (8,6.5) {156};
\draw [blue][-stealth](t) edge (v46);
\draw [blue][-stealth](v46) edge (v346);
\node (v34) at (9,5.5) {15};
\node (v134) at (10,6.5) {135};
\draw [blue][-stealth](v346) edge (v34);
\draw [blue][-stealth](v34) edge (v134);
\node (e) at (11,5.5) {e};
\draw [blue] [-stealth](v134) edge (e);
\node at (-2,6) {$F_{r:t\to e}:=$};
\node at (5,6) {$F_{b:t\to e}:=$};
\end{tikzpicture}
\end{center}

The possible flowlines $e\to 1$ are the following.

\begin{center}
\begin{tikzpicture}
\node (e) at (2,3) {e};
\node (v1) at (3,2) {1};
\draw [red][-stealth] (e) edge (v1);

\node (e) at (5,3) {e};
\node (v3) at (6,2) {3};
\draw  [blue][-stealth](e) edge (v3);
\node (v23) at (7,3) {23};
\node (v2) at (8,2) {2};
\node (v12) at (9,3) {12};
\node (v1) at (10,2) {1};
\draw  [blue][-stealth](v3) edge (v23);
\draw  [blue][-stealth](v23) edge (v2);
\draw  [blue][-stealth](v2) edge (v12);
\draw  [blue][-stealth](v12) edge (v1);
\node at (1,2.5) {$F_{r:e\to 1}:=$};
\node at (4,2.5) {$F_{b:e\to 1}:=$};
\end{tikzpicture}
\end{center}

Call the red former-half flowline $F_{r:t\to e}$ and the blue former-half flowline $F_{b:t\to e}$. Likewise, call the red latter-half flowline $F_{r:e\to 1}$ and the blue latter-half flowline $F_{b:e\to 1}$. We can compose either $t\to e$ flowline with either $e\to 1$ flowline to produce another critical path through the simplicial complex. Furthermore, applying the algorithm to $F_g=F_{r:t\to e}\circ F_{r:b\to 1}$ will give $F_y = F_{b:t\to e}\circ F_{r:e\to 1}$, shown in green and yellow respectively.
In this way, we can partition the simplex into two components, each bounded by two critical flow trajectories.

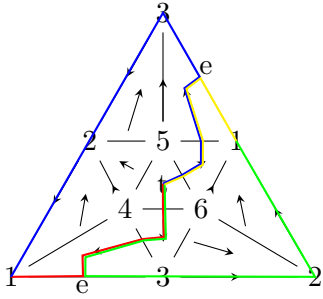
\begin{figure}[H]
    \centering
    \begin{tikzpicture}
\node (v2) at (2,0.5) {3};
\node (v1) at (0,-3) {1};
\node (v3) at (4,-3) {2};
\node (v4) at (1.035,-1.2) {2};
\node (v6) at (2.965,-1.2) {1};
\node (v5) at (2,-3) {3};
\node (v7) at (2,-1.2) {5};
\node (v8) at (1.5,-2.1) {4};
\node (v9) at (2.5,-2.1) {6};
\draw  (v7) edge (v2);
\draw  (v1) edge (v8);
\draw  (v9) edge (v3);
\draw  (v8) edge (v9);
\draw  (v9) edge (v7);
\draw  (v7) edge (v8);
\draw  (v4) edge (v7);
\draw  (v7) edge (v6);
\draw  (v6) edge (v9);
\draw  (v5) edge (v9);
\draw  (v8) edge (v5);
\draw  (v4) edge (v8);
\draw  (v2) edge (v4);
\draw  (v4) edge (v1);
\draw  (v1) edge (v5);
\draw  (v5) edge (v3);
\draw  (v3) edge (v6);
\draw  (v6) edge (v2);
\node (v21) at (1.4517,-0.4807) {};
\node (v20) at (0.4822,-2.1672) {};
\node (v126) at (3.3917,-1.9586) {};
\node (v18) at (0.8867,-2.464) {};
\node (v19) at (1.008,-1.7639) {};
\node (v16) at (3.1133,-2.464) {};
\node (v17) at (2.973,-1.774) {};
\node (v30) at (1.75,-1.65) {};
\node (v28) at (2.25,-1.65) {};
\node (v31) at (1.2893,-1.3894) {};
\node (v29) at (2.6967,-1.3844) {};
\node (v32) at (1.265,-1.65) {};
\node (v26) at (2.735,-1.65) {};
\node (v10) at (1.73,-2.5) {};
\node (v14) at (2.27,-2.5) {};
\node (v11) at (0.9415,-2.7169) {};
\node (v15) at (3.0647,-2.732) {};
\node (v13) at (2.0056,-2.6706) {};
\node (v12) at (2,-2.1) {};
\node (v22) at (1.5,-1.2) {};
\node (v24) at (2.5,-1.2) {};
\node (v23) at (1.7525,-0.37) {};
\node (v25) at (2.2436,-0.365) {};
\node (v27) at (3.1416,-3.0004) {};
\draw  [-stealth](v10) edge (v11);
\draw  [-stealth](v12) edge (v13);
\draw  [-stealth](v14) edge (v15);
\draw  [-stealth](v16) edge (v17);
\draw  [-stealth](v18) edge (v19);
\draw  [-stealth](v4) edge (v20);
\draw  [-stealth](v2) edge (v21);
\draw  [-stealth](v22) edge (v23);
\draw  [-stealth](v24) edge (v25);
\draw  [-stealth](v9) edge (v26);
\draw  [-stealth](v8) edge (v32);
\draw  [-stealth](v30) edge (v31);
\draw  [-stealth](v28) edge (v29);
\draw  [-stealth](v5) edge (v27);
\draw  [-stealth](v3) edge (v126);
\node (v111) at (2.0001,-0.2445) {};
\draw  [-stealth] (v7) edge (v111);
\node at (2.5657,-0.2259) {e};
\node at (0.9373,-3.1258) {e};
\node (v33) at (2.00,-1.7638) {t};
\node (v1111) at (2.0213,-1.7745) {};
\draw  [red,thick]plot[smooth, tension=0] coordinates {(v33) (v12) (2.0128,-2.4725) (v10) (v11) (0.9474,-2.9902) (v1)};
\draw  [blue,thick]plot[smooth, tension=0] coordinates {(v33) (v28) (2.5,-1.5) (v24) (2.2826,-0.4986) (2.4829,-0.3479) (v2) (v1)};
\draw  [green,thick]plot[smooth, tension=0] coordinates {(v1111) (2.037,-2.4945)  (1.7534,-2.5161) (0.9789,-2.7407) (0.9781,-2.9906) (v3) (v6)};
\draw  [yellow,thick]plot[smooth, tension=0] coordinates {(v1111) (2.2627,-1.6691) (2.5242,-1.5145) (2.5263,-1.1956) (2.3088,-0.5072) (2.4942,-0.3674) (v6)};
\end{tikzpicture}

    \caption{The figure shows the simplicial complex representing $\P^2$ partitioned by the critical flowlines, shown in red, blue, yellow and green.}
    \label{fig:RP2partitioned}
\end{figure}

Ultimately, we imagine that the real projective plane is this upper hemisphere, consisting of a bendy disk, a circle that serves as the equator, and a point on the circle to give us a groundedness. Defining an orientation on $t$ shows us that there are two paths $t\to e$ with the same induced sign, which tells us that $\delt(t)=2e$. However for the arrow from $e$ to the critical vertex 1, the sign on the arrow $t\to e$ induces a sign on $e\to 1$ from the left, and an opposite sign on $e\to 1$ from the right, so that they cancel out, and $\delt(e)=0$. 

We can discuss the homological implications of this. Let $K$ be the $\RPtwo$ simplicial complex. As there are two paths $t\to e$, the map from $C^{\morse}_{2}(K,f)$, the space generated by critical 2-simplices of $K$, to $C^{\morse}_{1}(K,f)$, the space generated by critical 1-simplices, is the map $\delt_1:\Z\to \Z:z\mapsto 2z$. Then, as there are two paths $e\to 1$ which cancel out in sign, the map $\delt_2:C_1(K)\to C_0(K)$ is the zero map. 

Calculating the homology groups in this way, we find that $H_2=\frac{\ker(\delt_2)}{\im(\delt_3)}=0$, since there is no $d_3$, $H_1=\frac{\ker(\delt_1)}{\im(\delt_2)}=\frac{\Z}{2\Z}$, and $H_0=\frac{\ker(\delt_1)}{\im(\delt_0)}=\Z$. 

\subsection{Orientation on floperations}\label{subsec:florientations}

We wish to show that every critical flowline has an opposite partner. To this end, let us restate the flip of the Flop.

\begin{lemma}[The flip of the Flop]
    For some path $P$, we must have
    \[\theta(\Flop(P))=-\theta(P).\]
\end{lemma}
\begin{proof}
    Follows from Lemma \ref{lem:flip_of_flop}.
\end{proof}

This is to say that each time we Flop a path, we negate its sign. Let us see the effect that the other floperations have on the sign.

\begin{lemma}[The flip of the Insert]
    For some path $P$, we must have
    \[\theta(P)=-\theta(\Insert(P)).\]
\end{lemma}
\begin{proof}
    The operation of Insertion gives us two new arrows, so that $\ell(\Insert(P))=\ell(P)+2$. Furthermore, they are both the same arrow, so will have the same sign. Therefore if $x$ is the arrow to be Inserted we have
    \begin{align*}
        \theta(\Insert(P))&=\left(\prod_{a\in P}\theta(a)\right)(\theta(x))^2\left((-1)^{\frac{\ell(P)+2-\Ind(P)}{2}}\right)\\
        &=(-1)\left(\prod_{a\in P}\theta(a)\right)\left((-1)^{\frac{\ell(P)-\Ind(P)}{2}}\right)\\
        &=(-1)\theta(P).
    \end{align*}
    This proves the desired equality.
\end{proof}

\begin{lemma}[The flip of the Cancel]
    For some path $P$, we must have
    \[\theta(P)=-\theta(\Cancel(P)).\]
\end{lemma}
\begin{proof}
    The Cancel operation removes both instances of the intermediate Morse arrow. This is to say that $\ell(\Cancel(P))=\ell(P)-2$ and both the arrows removed will have the same sign. Also for intermediate Morse arrow $x$, we have $\theta(x)^2=1$ since the sign can only ever be 1 or $-1$. So
    \begin{align*}
        \prod_{a\in P}\theta(a)&=\theta(x)^2\left(\prod_{a\in \Cancel(P)}\theta(a)\right)\\
        &=\prod_{a\in \Cancel(P)}\theta(a)
    \end{align*}
    Therefore we have that
    \begin{align*}
        \theta(\Cancel(P))&=\left(\prod_{a\in \Cancel(P)}\theta(a)\right)\left((-1)^{\frac{\ell(P)-2-\Ind(P)}{2}}\right)\\
        &=(-1)\left(\prod_{a\in P}\theta(a)\right)\left((-1)^{\frac{\ell(P)-\Ind(P)}{2}}\right)\\
        &=(-1)\theta(P),
    \end{align*}
    proving the desired equality.
\end{proof}
We conclude from these three lemmas that every floperation flips the sign on the path. This is to say that an odd number of floperations acting on $P$ will give the resulting path a sign of $-\theta(P)$, and an even number of floperations on $P$ will give the resulting path a sign of $\theta(P)$. 

In this section, we have defined the algorithm by which we generate the one-dimensional manifold of flowlines in a partition. We have analysed the application of the algorithm to two examples, observing how we can cycle through the different flowlines. We found how each floperation affects sign, and set ourselves up to show that the critical flowlines cancel each other out with opposite signs. In the next section, we indeed make such a definition, and prove hence that $\delt^2(\alpha)=0$.

\section{Proving that boundary flowlines come in pairs}\label{sec:Proving that the boundary square must have a pair}
The examples shown have given us an idea of how the algorithm works on simplicial complexes. In particular, one thing we can notice is that for some such generated list of flowlines $\AlgList_c(\FO)$ for $\FO$ critical, and some noncritical flowline label pair $(F_i,\ell_i)$ in the list, then $\AlgList_{\ell_i}(F_i)$ matches exactly the elements in $\AlgList_c(\FO)$ following $(F_i,\ell_i)$. 

Another thing worth noticing is that the algorithm is composed of floperations Cancel and Insert, which are co-inverse, and Flop, which is self-inverse. For this reason, we may suspect that the algorithm is its own inverse. We now prove that, for critical flowlines, this is true.

\subsection{Involutivity}

We wish to show that the algorithm is involutive, which is to say that the algorithm applied twice to a critical flowline will return the same flowline. 
\begin{lemma}[Algorithm involutivity]\label{lem:algInvolutivity}
    Let $\FO$ be a critical flowline, and $\Fn=\Alg(\FO)$, where the length of the list $[\FO,F_1,\dots,F_{n-1},\Fn]$ of flowlines is $n+1$.
    The algorithm is involutive when applied to any critical flowline $\FO$, so that applying the algorithm to $\Fn$ gives a list $[\Fn,F_{n+1},\dots,F_{n+m-1},\FO]$. If the $(n-i)$th flowline in the list is $F_{n-i}$, then we have $F_{n-i}=F_{n+i}$ (where the labels are not necessarily the same). Furthermore, if $F_{n+i}$ has label $\ell$ then $F_{n-i}$ has the conjugate label $\overline{\ell}$. 
     
\end{lemma}
\begin{proof}
    Let $F_0$ be a critical flowline, and assume that the algorithm returns the list of flowlines $\F=[(\widetilde{F_0},\ell_0),(F_1,\ell_1),\dots,(\widetilde{F_n},\ell_n)]$, where the $\widetilde{F_i}$ are not assumed to be unique\footnote{We will later show that we cannot have $\FO=\Fn$.}. Denote by $\G$ the list
    given by applying the algorithm to $\GO\coloneq\Fn$, so that $\G=[(\GO,\mathfrak{l}_0),(G_1,\mathfrak{l}_1),\dots,(\Gm,\mathfrak{l}_m)]$. Since the algorithm terminates only at critical flowlines, we acknowledge that $\GO$ and $\Gm$ must be critical. 
    
    Since Flop is self-inverse, and Insert and Cancel are inverse to each other, we claim that applying the algorithm to $F_n$ is akin to reversing all our steps to get there. To evidence this, we examine each of the steps of the algorithm individually, to ultimately show by induction that $F_{n-i}=G_i$ for all natural $i\leq n$, and indeed that $\FO=\Gm=\Fnm$.

    We wish to show 
    \begin{enumerate}
        \item $(F_{n},\ell_{n})=(G_0,\overline{\ell_0})$ (the base case).
        \item if $(F_{n-i},\ell_{n-i})=(G_i,\overline{\ell_i})$ then $(F_{n-i-1},\ell_{n-i-1})=(G_{i+1},\overline{\ell_{i+1}})$ (the induction). 
    \end{enumerate}
    For the induction step, we assume that $(F_{n-i},\ell_{n-i})=(G_i,\overline{\ell_i})$. We prove the following.
    \begin{itemize}
        \item If $G_i$ has label `f' then:
        \begin{enumerate}[label = (\alph*)]
            \item if $\Flop(\Insert(G_i))$ is legal then $G_{i+1}=\Flop(\Insert(G_i))$ and $F_{n-i-1}=G_{i+1}$.
            \item if $\Flop(\Insert(G_i))$ is illegal then $G_{i+1}=\Cancel(\Flop(\Insert(G_i)))$and $F_{n-i-1}=G_{i+1}$.
        \end{enumerate}
        \item Alternatively, if $G_i$ has label `c' then: 
        \begin{enumerate}[resume*]
            \item if $\Flop(G_i)$ is legal then $G_{i+1}=\Flop(G_i)$ and $F_{n-i-1}=G_{i+1}$.
            \item if $\Flop(G_i)$ is illegal then $G_{i+1}=\Cancel(\Flop(G_i))$ and $F_{n-i-1}=G_{i+1}$.
    \end{enumerate}
    \end{itemize}
    \begin{figure}[H]
        \centering
        \begin{tikzpicture}
\node[inner sep=0,outer sep=0] (v1) at (-2,2) {};
\node at (2,2) {};
\node[inner sep=0,outer sep=0] (v3) at (-2,-2) {};
\node[inner sep=0,outer sep=0] (v2) at (2,-2) {};
\draw [step=2cm,black] (v1) grid (v2);
\node at (-1,2.5) {$\ell_i=f$};
\node at (1,2.5) {$\ell_i=c$};
\node at (-3.5,1) {$\ell_{i+1}=f$};
\node at (-3.5,-1) {$\ell_{i+1}=c$};
\node at (-1,1) {(a)};
\node at (1,1) {(c)};
\node at (-1,-1) {(b)};
\node at (1,-1) {(d)};
\draw  (v1) edge (v3);
\draw  (v3) edge (v2);
\end{tikzpicture}
        \caption{The figure shows the four cases to prove in the induction step to prove the hypothesis.}
        \label{fig:InvolutivityInductionCases}
    \end{figure}
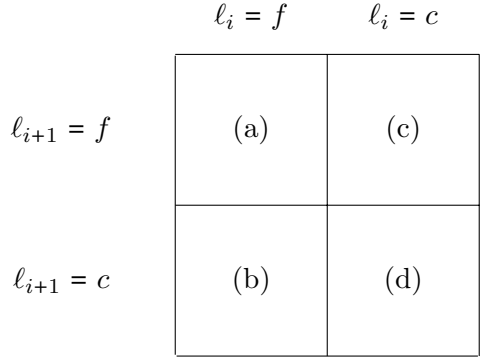
    In proving these steps, we will show first that the base case is true, and prove by induction that any of the four possible steps from some Append to another Append, (a), (b), (c), or (d), must yield the desired equality.
    
    First, we notice that $\widetilde{G_0}=\Fn$ by assumption. By the definition of the algorithm, the first flowline $\widetilde{G_0}$ always has label `c', and the last flowline $\Fn$ always has label `f'. Therefore $(\Fn,\ell_n)=(\widetilde{G_0},\overline{\ell_0})$, proving the base case.
    
    We now analyse the point of the algorithm at which Flop takes us to a legal path (given by label `f'), causing us to branch out into the left-hand side of the algorithm flowchart. So, assume that $F_{n-j}=G_j$, and $G_j$ has label `f'.
    \begin{enumerate}[label = (\alph*)]
        \item If $\Flop(\Insert(G_{j}))$ is legal then $G_{j+1}=\Flop(\Insert(G_{j}))$. Furthermore, assuming that $F_{n-j}=G_{j}$, and using the fact that Flop was the last operation to get to $G_j$, we know that the first floperation from $F_{n-j}$ to $F_{n-j+1}$ is Flop, so the previous operation cannot be Flop. So the floperation directly before $F_{n-j}$ must be Cancel (since Insert would make $F_{n-j}$ illegal). Before Cancelling, we cannot have the operation Cancel as that would imply the existence of a second distinct dimension jump adjacent to the critical simplex, and we cannot have an Insertion as $\Insert(\Cancel)=\id$. Then the previous operation must be a Flop. Furthermore, the algorithm shows us that the only way we can Flop a legal flowline is if it has the label `c' (this is because if it has label `f' then we apply Flop twice, which is trivial, and therefore impossible).
        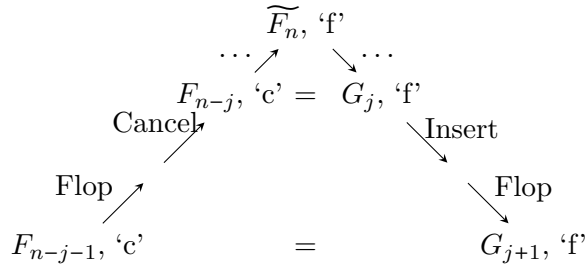
\begin{figure}[H]
            \centering
            \begin{tikzpicture}

\node (v3) at (0,2) {$\widetilde{F_n}$, `f'};
\node (v4) at (1,1) {$G_{j}$, `f'};
\node (v5) at (2,0) {};
\node (v2) at (-1,1) {$F_{n-j}$, `c'};
\node (v1) at (-2,0) {};
\draw [-stealth] (v1) edge (v2);
\draw [-stealth] (v2) edge (v3);
\draw [-stealth] (v3) edge (v4);
\draw [-stealth] (v4) edge (v5);
\node at (0,1) {=};
\node at (-1.9794,0.6711) {Cancel};
\node at (-0.8928,1.5618) {$\dots$};
\node at (0.9777,1.5618) {$\dots$};
\node at (2.0555,0.5999) {Insert};
\node (v7) at (3,-1) {$G_{j+1}$, `f'};
\node (v6) at (-3,-1) {$F_{n-j-1}$, `c'};
\draw  [-stealth](v6) edge (v1);
\draw  [-stealth](v5) edge (v7);
\node at (-2.9141,-0.1928) {Flop};
\node at (2.8664,-0.2106) {Flop};
\node at (0,-1) {=};
\end{tikzpicture}
            \caption{The figure shows the operations of Insert and Flop on some $G_j$, and the inverse operations Flop and Cancel on $F_{n-j-1}$, showing that $F_{n-j-1}=G_{j+1}$.}
            \label{fig:insertflopgiesequality}
        \end{figure}
        This is to say that 
        \begin{align*}
            G_{j+1}&=\Flop(\Insert(\Cancel(\Flop(F_{n-j-1}))))\\
            &=\Flop(\Flop(F_{n-j-1}))\\
            &=F_{n-j-1},
        \end{align*}
        and the labels of $F_{n-i}$ and $G_i$ are conjugate, as required. Then the hypothesis is true for case (a) in the table.
        \item If $\Flop(\Insert(G_{j}))$ is \textit{not} legal, then we Cancel as per the algorithm, so that $G_{j+1}$ has label `c'. On the left, we must have some operation before Flopping and Cancelling. We cannot Flop directly before a Flop as the operation is self-inverse, and we cannot Cancel since we assumed that the flowline after said operation is not legal: so it must have been Insert. From this, we can also see from the algorithm that $F_{n-j-1}$ must have label `f', since if it were preceded by a Cancel we would have $\Insert\circ\Cancel=\id$. We find that 
        \begin{align*}
            G_{j+1}&=\Cancel(\Flop(\Insert(\Cancel(\Flop(\Insert(F_{n-j-1}))))))\\
            &=\Cancel(\Flop(\Flop(\Insert(F_{n-j-1}))))\\
            &=\Cancel(\Insert(F_{n-j-1})\\
            &=F_{n-j-1}
        \end{align*}
        and the label of $G_{j+1}$ is conjugate to that of $F_{n-j-1}$. This proves the hypothesis for case (b) in the table.
    \end{enumerate}

Therefore, given $F_{n-i}=G_{i}$ where $G_{i}$ has label `f', we see that both the step via (Insert, Flop) and the step via (Insert, Flop, Cancel) induce $F_{n-i-1}=G_{i+1}$, proving the hypothesis for the left column of cases in the table. 

For the right column of the table, we analyse the point of the algorithm at which Cancel takes us to a legal path (given by label `c'). So assume that $F_{n-k}=G_k$ and $G_k$ has label `c'.
\begin{enumerate}[resume*]
    \item By the assumption on the label relation, we assume that $F_{n-k}$ has label `f'. If $\Flop(G_k)$ is legal, we need only note that 
    \[F_{n-k-1}=\Flop(F_{n-k})=\Flop(G_{k})=\Flop(G_{k+1}).\]
    To prove induction on the label hypothesis, note that $G_k$ has label `f', and as the step before Flopping, $G_{n-k}$ has label `c'. So the hypothesis is true for case (c) in the table.
    \item If $\Flop(G_k)$ is illegal, we must Cancel to get to a legal path. To see what's going on on the other side, we realise that the last operation from $F_{n-k-1}$ to $F_{n-k}$ must still be Flop, and the floperation preceding it, therefore, must either be Cancel or Insert. However Cancel would obtain a legal path, and we have assumed that $\Flop(F_{n-k})$ is not legal, so the operation must be Insert. Since Insert can only be applied to a legal path, we must have $F_{n-k}=\Flop(\Insert(F_{n-k-1}))$, giving us 
    \begin{align*}
        G_{k+1}&=\Cancel(\Flop(\Flop(\Insert(F_{n-k-1}))))\\
        &=\Cancel(\Insert(F_{n-k-1}))\\
        &=F_{n-k-1},
    \end{align*}
    as required. We have therefore proved case (d) in the table.
\end{enumerate}

We have analysed all the different possible cases, and thus have proved that $F_{n-1}=G_{1}$, and by induction that $F_{n-i}=G_{i}$ for all $i\leq n$. This is to say that $\FO=F_{n-n}=\Gm=F_{n+m}$, which is critical and therefore terminating, i.e. $m=n$ and $\FO=\Fnn$.
\end{proof}

Given that the algorithm is involutive, we can say more about the flowline equivalence we have defined.
\begin{lemma}\label{lem:equiv.rel}
    The flowline equivalence given by Definition \ref{def:flowlineEquivalnece} associating flowlines in the same algorithm list with each other is an equivalence relation $\sim$ in that, for any flowlines $F_i, F_j, F_k$ through a given simplicial complex, we have reflexivity ($F_i\sim F_i$); symmetry ($F_i\sim F_j$ implies $F_j\sim F_i$); and transitivity ($F_i\sim F_j$ and $F_j\sim F_k$ implies $F_i\sim F_k$).
\end{lemma}

\begin{proof}
We show that $\sim$ is reflexive, symmetric and transitive.
\begin{itemize}
    \item $F_i\sim F_i$ : Let $\F = [\FO,\dots,F_i,\dots,F_j,\dots,\Fn]$ be the list of flowlines in the algorithm applied to a critical flowline in the same list as $F_i$. Since $F_i\in \F$, then $F_i\sim F_i$.
     
    \item $F_i\sim F_j\Rightarrow F_j\sim F_i$ : Let $\F = [\FO,\dots,F_i,\dots,F_j,\dots,\Fn]$ be the list of flowlines in the algorithm applied to a critical flowline in the same list as $F_i$. Then applying the algorithm to $\Fn$, we see that $F_j$ is in the same list as $F_i$. Therefore $F_j\sim F_i$.
     
    \item $F_i\sim F_j$ and $F_j\sim F_k\Rightarrow F_i\sim F_k$ : Let $\F = [\FO,\dots,F_i,\dots,F_j,\dots,F_k,\dots,\Fn]$ be the list of flowlines in the algorithm applied to a critical flowline in the same list as $F_i$. 
    Clearly, $F_i$ is in the same list as $F_k$, so $F_i\sim F_k$.
     
\end{itemize}
\end{proof}
We now wish to show that there are an even number of boundary flowlines in $\M(\alpha,\gamma)$, and thirdly that for every boundary flowline, there is a unique distinct boundary flowline, with the opposite sign, attained by the algorithm.

\subsection{Boundary flowlines have unique distinct partners}\label{subsec:Boundary flowlines have unique distinct partners}
We now wish to show that $\delt^2(\alpha)=0$, regardless of the critical $(p-1)$-simplices in the simplicial complex. 
\begin{lemma}\label{lem:2critFsPerEquivClass}
    For $\alpha$ an $(n+1)$-simplex and $\gamma$ an $(n-1)$-simplex, the moduli space of flowlines $\overline{\M}(\alpha,\gamma)$ is a simplicial manifold with boundary of index 1 (i.e. a disjoint union of paths and cycles).
\end{lemma}
The above lemma is equivalent to the statement that for each equivalence class of flowlines (i.e. connected component of $\M(\alpha,\gamma)$), there are either exactly two critical flowlines (an interval has two edgepoints) or zero critical flowlines (a cycle has no edgepoints).
 
\begin{proof}
    We show that boundary flowlines always come in pairs. Clearly, the algorithm partitions flowlines such that there can be at most two critical flowlines per equivalence class. We have also seen from the sphere example (\ref{subsec:sphere_ex}) that when there are no critical flowlines there are no boundary flowlines and there is one equivalence class of flowlines defined by the algorithm, which is to say that there can sometimes be no critical flowlines. Since this case adds no new terms to the sum of critical flowlines, we need not consider such cases. 
    Thus, we can reduce the problem to showing that:
    \begin{enumerate}[label = (\alph*)]
        \item if there exists a critical flowline in an equivalence class, then applying the algorithm will yield a distinct critical flowline, and
        \item any noncritical $F$ in the moduli space of flowlines is adjacent via the algorithm to exactly two flowlines.
    \end{enumerate}
    
    We prove that the two critical flowlines belonging to each connected component of the moduli space of flowlines must be distinct. We find that as Cancel and Insert are inverse actions and Flop is self-inverse, the starting steps of the algorithm will mirror the concluding steps of the algorithm. This is to say that $F_1=F_{n-1}$, $F_2=F_{n-2}$, and ultimately $F_i=F_{n-i}$. Therefore there must be a turning point in the flowline algorithm: i.e. 
    \begin{itemize}
        \item if $n$ is odd then $n=2k+1$ for some $k$, and $F_{k}=F_{k+1}$.
        \item if $n$ is even then $n=2k$ for some $k$, and $F_{k-1}=F_{k+1}$.
    \end{itemize}
    We intend to show that neither of these is possible. 
    
    Indeed it is clear that applying Insert, Flop and (Cancel and Flop while path is illegal) to $F_k$ will have a nontrivial effect on the flowline.
    This is because we have inserted a simplex that was not in $F_k$, and flopping means that the inserted simplex cannot be the simplex we cancel, i.e. we cannot have $F_k=F_{k+1}$.  

    For the second case, where $n$ is even, we recall from Lemma \ref{lem:algInvolutivity} that if $\FO$ and $\Fn$ are critical with $n+1$ flowlines in $\AlgList(\FO)$ then we have that the label of $F_{n-i}$ is the conjugate of the label of $F_{n+i}$. If $\FO=\Fn$, then we must have that $\ell_1=\overline{\ell_{n-1}}$, $\ell_2=\overline{\ell_{n-2}}$, \dots, $\ell_{n/2}=\overline{\ell_{n/2}}$. Then $\ell_{n/2}\neq \ell_{n/2}$, and we have a contradiction.

    Therefore, since in no case can we have $F_{k-1}=F_{k+1}$ or $F_{k}=F_{k+1}$, we find that there can never be a unique critical flowline in a partition. This is to say that for each $\gamma\in \Gamma$, for a boundary flowline $F_{a}:\alpha\to \gamma$ there must be a distinct boundary flowline $F_{b}:\alpha\to \gamma$.

    Now we prove (b). Suppose for contradiction that $F$, a noncritical flowline, is adjacent to more than two distinct flowlines. Then $\deg(F):=n\geq 2$. Call them the distinct flowlines $A_1$, $A_2$, $A_3$ et cetera. For each adjacent flowline, we can apply the algorithm to each $A_i$ to obtain $n$ distinct critical flowlines $\widetilde{A_1},\widetilde{A_2},\dots$. Assume without loss of generality that $\AlgList_c(\widetilde{A_1})$ terminates at $\widetilde{A_2}$, and $\AlgList_c(\widetilde{A_2})$ terminates at $\widetilde{A_3}$. Then by the involutivity proven in \ref{lem:algInvolutivity} we must have that $\widetilde{A_1}=\widetilde{A_3}$. Indeed, $\AlgList_c$ generates the set of critical flowlines, where 
    \[(\AlgList_c)^n(\widetilde{A_1})=\begin{cases}
        \widetilde{A_1} \quad &\text{for }n\text{ even}\\
        \widetilde{A_2} \quad &\text{for }n\text{ odd,}\\
    \end{cases}\]
    contradicting the distinctiveness of the $\tilde A_i$. Indeed, we can conclude from this that $F$ has exactly two critical points related to it via the equivalence relation, and therefore has exactly two adjacent flowlines $A_1$ and $A_2$.
\end{proof}

This proof tells us that if a flowline is repeated then there has been a full cycle around possible flowlines without passing a critical point, so turning back isn't a floption.  

We now prove that all boundary flowlines of the moduli space can be partitioned into their equivalence classes.
\begin{lemma}\label{lem:del=boundarysum2}
    For a critical simplex $\alpha^{(p+1)}$, the set $\Gamma=\{\gamma_1,\dots,\gamma_n\}$ of $p-1$ dimensional critical simplices, the set $\F(\alpha,\gamma)$ of critical flowlines, and $\E(\alpha,\gamma)$ the set of representatives of equivalence classes $[F_i]$ with boundaries, we have
    \begin{align*}
    \del\M(\alpha,\beta)=\sum_{[F_i]\in\E(\alpha,\gamma)}\left(\theta(\Alg_c(F_i))+\theta(\Alg_c(\Alg_c(F_i)))\right).
    \end{align*}
\end{lemma}
\begin{proof}
    Let $F_-\in\F(\alpha,\gamma)$. Then there is some $[F_i]\in\E$ such that $F_-\sim F_i$. This is to say that either $F_-=\Alg_c(F_i)$ or $F_-=\Alg_c(\Alg_c(F_i))$. Without loss of generality, assume that $F_-=\Alg_c(F_i)$, and denote $F_+:=\Alg_c(\Alg_c(F_i))$. Then we can replace the terms of these two flowlines in the sum $\sum_{F\in \F(\alpha,\gamma)}\theta(F)\cdot\gamma$ by the single term $\theta(F_-)\cdot\gamma+\theta(F_+)\cdot\gamma$. Since every $F\in \F(\alpha,\gamma)$ belongs to some equivalence class, we can partition the entire sum into pairs in this way. Then
    \[\sum_{\gamma \in\Gamma}\sum_{F\in \F(\alpha,\gamma)}\theta(F)\cdot\gamma=\sum_{\gamma\in\Gamma}\sum_{[F_i]\in\E(\alpha,\beta)}\left(\theta(\Alg_c(F_i))\cdot\gamma+\theta(\Alg_c(\Alg_c(F_i)))\cdot\gamma\right),\]
    as required.
\end{proof}

We now prove that there are always an odd number of floperations to get from one critical flowline to another.

\begin{lemma}\label{1mod2floperations}
    There are $1\mod 2$ many floperations throughout the duration of the algorithm applied to a critical flowline.
\end{lemma}
\begin{proof}
    By analysing the algorithm, we can see that the first step is always to Flop, which is 1 floperation modulo 2. Not considering the bookkeeping, there are two options from here. We can either Cancel and Flop, or Insert and Flop. As the algorithm is composed of some number of iterations of these two-floperation chunks, we must always have performed 1 floperation modulo 2 after applying a Flop. Furthermore, the algorithm can only terminate after having Flopped. Therefore there must be 1 modulo 2 floperations in the algorithm from some critical flowline to another.
\end{proof}
\begin{corollary}\label{cor:theta.alg=-theta.alg.alg}
    For $\alpha,\gamma$ critical simplices with $\dim(\alpha)=\dim(\gamma)+2$, let $I$ be a connected component of $\M(\alpha,\gamma)$. Then for a flowline $F\in I$ if the algorithm terminates at some critical flowline, we have
    \[\theta(\Alg_c(F))+\theta(\Alg_c(\Alg_c(F)))=0\]
\end{corollary}
\begin{proof}
    Let the critical flowline $\Alg_c(F)$ found by applying the algorithm to $F$ have sign $\Theta$. There must be 1 floperation modulo 2 to get from $\Alg_c(F)$ to its unique, distinct partner $\Alg_c(\Alg_c(F))$. Since each floperation flips the sign of the flowline, the terminating critical flowline will have sign $(-1)^{2n+1}\cdot\Theta=-\Theta$. This is to say that $\theta(\Alg_c(\Alg_c(F)))=-\theta(\Alg_c(F))$, proving the corollary.
\end{proof}
 
We can now state the theorem we have been chasing.
\begin{theorem}[Explicit cancellation of flowlines]\label{thm:explicitCancellationOfFlowlines}
For a critical $(n+1)$-simplex $\alpha$, $\Gamma$ the set of critical $(n-1)$-simplices, and $\F(\alpha,\gamma)$ the set of critical flowlines from $\alpha$ to $\gamma$, we have
    \begin{align*}
        \sum_{\gamma\in\Gamma}\sum_{\widetilde{F_i}\in \F(\alpha,\gamma)}\theta\left(\widetilde{F_i}\right)\cdot \gamma =0.
    \end{align*}
\end{theorem}
\begin{proof}
    For the equality,
    \[\sum_{\gamma\in\Gamma}\#(\text{boundaries of flowlines } \alpha\to\gamma \text{ with 2 drops})\cdot \gamma=\sum_{\gamma\in\Gamma}\sum_{\widetilde{F_i}\in \F(\alpha,\gamma)}\theta\left(\widetilde{F_i}\right)\cdot \gamma,\]
    we know that boundary flowlines are exactly critical flowlines $\widetilde{F_i}\in \F(\alpha,\gamma)$.
    Furthermore, we can see that two boundary flowlines $\widetilde{F_i},\widetilde{F_j}$ from $\alpha$ to some $\gamma\in \Gamma$ will cancel each other out if they have opposite signs. This is to say that if there are 0 critical flowlines modulo 2 then the critical flowlines that exist must cancel in sign. 

    To prove the equality \[\sum_{\gamma\in\Gamma}\sum_{\widetilde{F_i}\in \F(\alpha,\gamma)}\theta\left(\widetilde{F_i}\right)\cdot \gamma =0\]
    we only need to show that 
    \[\sum_{\widetilde{F_i}\in \F(\alpha,\gamma)}\theta\left(\widetilde{F_i}\right)=0\]
    for each $\gamma$. To simplify this further, we recognise that each critical flowline $\widetilde{F_{+}}$ comes with a unique distinct critical flowline $\widetilde{F_{-}}$ in its equivalence class. Indeed, by Corollary \ref{cor:theta.alg=-theta.alg.alg} we know that every connected component of $\M(\alpha,\beta)$ has one critical flowline of sign (+1) and one of sign (-1). Using this, we partition $\F(\alpha,\gamma)$ into two sets $\F_{+}\coloneq\{F\in\F(\alpha,\gamma):\theta(F)=(+1)\}$ and $\F_{-}\coloneq\{F\in\F(\alpha,\gamma):\theta(F)=(-1)\}$. This is to say that if $|\F(\alpha,\gamma)|=2k$ then 
    \begin{align*}
        \sum_{\gamma\in\Gamma}\sum_{\widetilde{F}\in \F(\alpha,\gamma)}\theta\left(\widetilde{F}\right)\cdot \gamma&=\sum_{\gamma\in\Gamma}\sum_{\widetilde{F}\in \F_{+}}\theta\left(\widetilde{F}\right)\cdot\gamma+\sum_{\gamma\in\Gamma}\sum_{\widetilde{F}\in \F_{-}}\theta\left(\widetilde{F}\right)\cdot\gamma\\
        &= \sum_{\gamma\in\Gamma} k\cdot(+1)+\sum_{\gamma\in\Gamma} k\cdot(-1)\\
        &=\sum_{\gamma\in\Gamma} k-k\\
        &=0.
    \end{align*}
    This proves the equality.
\end{proof}

This proves that $\delt^2(\alpha)=0$ for any critical simplex $\alpha$, thus concluding the paper.

\printbibliography 
\appendix

\end{document}